\documentclass[final]{siamltex704}  
\usepackage{theorem}
\usepackage[ruled]{algorithm2e} % ubuntu package: texlive-science 

\usepackage{url} 
\usepackage{tikz}
\usepackage{amsmath}

\usepackage{graphicx}
\usepackage{psfrag}
\usepackage{amssymb}
\usepackage{amsmath}
\usepackage{subcaption}

\renewcommand{\ldots}{\dotsc}

\usepackage{color}

\newcommand{\ee}{{\rm e}\hspace{1pt}}
\newcommand{\dd}{{\rm d}}
\newcommand{\ii}{\text{i}\hspace{1pt}}

\newcommand{\wt}{\widetilde}

\newcommand{\abs}[1]{\left| #1 \right|}

\newcommand{\trans}[1]{{#1^{\mathrm T}}}

\newcommand{\veps}{\varepsilon}

\newcommand{\RR}{\mathbb{R}}

\newcommand{\CC}{\mathbb{C}}

\newcommand{\NN}{\mathbb{N}}

\newcommand{\norm}[1]{||#1||}

\newtheorem{thm}{Theorem}

\newtheorem{lem}[thm]{Lemma}

\newtheorem{remark}[thm]{\textit{Remark}}

\title{The infinite Arnoldi exponential integrator for linear inhomogeneous ODEs}

\author{Antti Koskela, Elias Jarlebring
\thanks{Dept. Mathematics, KTH Royal Institute of Technology, SeRC Swedish e-Science Research Center, 
Lindstedtsv\"agen 25, Stockholm, Sweden, email: \{akoskela, eliasj\}@kth.se}}

\begin{document}
\maketitle

\begin{abstract}
Exponential integrators that use Krylov approximations of matrix functions 
have turned out to be efficient for the time-integration of certain ordinary 
differential equations (ODEs). This holds in particular
for linear homogeneous ODEs, where the exponential integrator
is equivalent to approximating the product of the matrix exponential and a vector.
In this paper, we consider  linear inhomogeneous ODEs, $y'(t)=Ay(t)+g(t)$,
where the function $g(t)$ is assumed to satisfy certain regularity
conditions. We derive an algorithm for this problem which is equivalent to
approximating the product of the matrix exponential and a vector using
Arnoldi's method. The construction is based on 
expressing the function $g(t)$ as a linear combination
of given basis functions $[\phi_i]_{i=0}^\infty$ 
with particular properties. The properties
are such that the inhomogeneous ODE can be restated
as an infinite-dimensional
linear homogeneous ODE. Moreover, the linear homogeneous
infinite-dimensional ODE has properties that directly allow us to extend
a Krylov method for finite-dimensional linear ODEs. Although the
construction is based on an infinite-dimensional
operator, the algorithm can be carried out with operations
involving matrices and vectors of finite size.
This type of
construction resembles
in many ways the infinite Arnoldi method for nonlinear eigenvalue
problems \cite{Jarlebring:2012:INFARNOLDI}. We prove
convergence of the algorithm under certain natural conditions, and 
illustrate properties of the algorithm with  examples stemming from the
discretization of partial differential equations.
\end{abstract}



\begin{DOI}

\end{DOI}

%%%%%%%%%%%%%%%%%%%%%%%%%%%%%%%%%%%%%%%%%%%%%%%%%%%%%%%%%%%%%%%%%%%%%%%%%%%%%%%%%%%%%%%%%%%%%%
\section{Introduction}
%%%%%%%%%%%%%%%%%%%%%%%%%%%%%%%%%%%%%%%%%%%%%%%%%%%%%%%%%%%%%%%%%%%%%%%%%%%%%%%%%%%%%%%%%%%%%%
Consider a matrix  $A\in\CC^{n\times n}$ and a function
$g: \CC \rightarrow \mathbb{C}^n$ with elements which are 
entire functions.
We consider the problem of numerically computing
the time-evolution of the linear ordinary differential 
equation with an inhomogeneous term
%For the time integration of linear systems of stiff differential equations
%\begin{subequations}
\begin{equation}\label{eq:semilinearODE}
u'(t) = A u(t) + g(t), \quad u(0) = u_0.
%u'(t) = f(t,u(t)), \quad u(0) = u_0,
\end{equation}
%where $f:\RR\times\CC^n\rightarrow\CC^n$ corresponds to a linear inhomogeneous problem
%\begin{equation}
%f(t,u(t))= A u(t) + g(t),
%\end{equation}
%\end{subequations}
%where $A\in\CC^{n\times n}$ and 
%$g: [0,\infty) \rightarrow \mathbb{C}^n$ is assumed to be analytic.
%we consider certain type of exponential integrators. 
Our focus will be on equations that arise from spatial
semidiscretization of partial differential equations 
of evolutionary type, 
and $A$ will typically be a large sparse matrix,
and $g$ will 
%in certain sense 
neither be close to linear nor correspond to an extremely stiff nonlinearity,
in a sense which is further explained in the examples in Section~\ref{sect:examples}.

The general problem of computing the time-evolution of ODEs can be approached
with various numerical methods.  The method we will present in this paper 
belongs to the class of methods  called \emph{exponential
integrators}. Exponential integrators have recently received considerable interest;
%and they have been adapted for large-scale (linear or semilinear) ODEs
 see the review paper \cite{Hochbruck:2010:EXPINT}.
An attractive feature of these methods stems from the combination of approximation of matrix functions
and the use of Krylov methods \cite{Hochbruck:1998:EXPINT}. 
This is mostly due to the superlinear convergence of the Krylov approximation of entire
matrix functions \cite{Hochbruck:1998:KRYLOVERROR}.
%Under appropriate conditions, exponential integrators
%with Krylov approximation
%lead to superlinear convergence (as a function
%of the number of matrix-vector products). 

In this paper we will present a new exponential integrator for \eqref{eq:semilinearODE}. The integrator is  constructed
using a particular type of expansion
of  the function $g$ in \eqref{eq:semilinearODE}. We will consider expansions
of the type 
\begin{equation} \label{eq:expansion}
g(s) = \sum\limits_{\ell = 0}^\infty w_\ell \phi_\ell(s),
\end{equation}
where $w_\ell \in \mathbb{C}^n$, $\ell \in \mathbb{N}$, and the basis 
functions $\phi_0,\phi_1,\dots$ are assumed to satisfy 
%the  properties:
\begin{subequations}\label{eq:inf_system}
\begin{eqnarray} 
\frac{\dd}{\dd t} \begin{bmatrix} \phi_0(t) \\ \phi_1(t) \\ \vdots \end{bmatrix} &=& H_\infty \begin{bmatrix} \phi_0(t) \\ \phi_1(t) \\ \vdots \end{bmatrix}\label{eq:inf_system1}, \quad \quad \begin{bmatrix} \phi_0(0) \\ \phi_1(0) \\ \vdots \end{bmatrix}=
e_1%\\
%\begin{bmatrix} \phi_0(0) \\ \phi_1(0) \\ \vdots \end{bmatrix}&=&
%e_1,\label{eq:inf_system2}
%%\begin{cases} \phi_0(0) &= 1, \\ \phi_\ell(0) &= 0, \, \ell > 0 \end{cases},
\end{eqnarray}
where $H_\infty\in\RR^{\infty \times \infty}$ is an infinite-dimensional Hessenberg matrix,
satisfying for a fixed constant $C\ge 0$, 
\begin{equation}
\|H_N\|< C\textrm{ for all }N=0,\ldots,\infty. \label{eq:HNasm}
\end{equation}
\end{subequations}
The matrix $H_N\in\RR^{N\times N}$ is the leading submatrix of $H_\infty$.

The scaled monomials is the easiest example
of such a sequence of functions. If we define $\phi_\ell(t):=t^\ell/\ell!$, $i=0,\ldots$, 
then \eqref{eq:inf_system} is satisfied with $H_\infty$ given by a transposed Jordan matrix
\begin{equation} \label{eq:H_jordan}
H_\infty = \begin{bmatrix}
   0 & & & \\
   1 &0 & & \\
    &1 &0 & & \\
   & &\ddots & \ddots &
 \end{bmatrix}.
\end{equation}
 In this case, the
expansion \eqref{eq:expansion} corresponds to a Taylor expansion and
the coefficients are given by $w_\ell=g^{(\ell)}(0)$, $\ell=0,\ldots$.
We will also see  that these properties
are satisfied for other functions, e.g., the Bessel function and the modified Bessel function of the first kind
(as we will further explain in Section~\ref{sect:otherexps}). 
The algorithm will be derived and analyzed for these choices of $[\phi_i]_{i=0}^\infty$.
The choice of basis functions can be tailored for the problem, and the best choice is problem dependent. 
This will be illustrated in the numerical examples in Section~\ref{sect:examples}.

%If the expansion \eqref{eq:expansion} is truncated at $\ell=N$, yielding
The general idea of our approach can be seen as follows. 
If  $\phi_0,\phi_1,\dots$ are the scaled monomials, then we can 
truncate \eqref{eq:expansion} at $\ell=N$, yielding
$\tilde{y}'(t)=A\tilde{y}(t)+\sum_{\ell=0}^{N-1}w_\ell \phi_\ell(t)$
and it 
straightforward
to verify that the inhomogeneous ODE \eqref{eq:semilinearODE}
can be expressed as a larger linear homogeneous ODE,
\begin{equation}\label{eq:homogenODE}
\frac{d}{dt}
\begin{bmatrix}
\tilde{y}(t)\\
\phi_0(t)\\
\vdots\\
\phi_{N-1}(t)\\
\end{bmatrix}
=A_N
\begin{bmatrix}
\tilde{y}(t)\\
\phi_0(t)\\
\vdots\\
\phi_{N-1}(t)\\
\end{bmatrix},
\;\;
\begin{bmatrix}
y(0)\\
\phi_0(0)\\
\vdots\\
\phi_{N-1}(0)\\
\end{bmatrix}=
\begin{bmatrix}y_0\\e_1
\end{bmatrix}
\end{equation}
where we have defined
\begin{equation}\label{eq:Adef}
A_N:=\begin{bmatrix}
A&W_N \\
0&H_N 
\end{bmatrix}
\end{equation}
and $H_N\in\RR^{N\times N}$ is the leading  $N\times N$ block of $H_\infty$ 
and $W_N:=[w_0\;\cdots\;w_{N-1}]\in\CC^{n\times N}$.
This relation has been used in \cite[Theorem~2.1]{AlMohy:2011:EXPINT} and 
also in \cite{Koskela:2013:EXPONENTIALTAYLOR}. 
If we combine this type of construction with an iterative
method (in a particular way),  we will here be able to 
construct an algorithm for \eqref{eq:semilinearODE} for 
 any sequence of functions 
$\phi_0,\phi_1,\dots$
satisfying \eqref{eq:inf_system}.
%Therefore a given sequence of nested Hessenberg matrices (bounded norm) gives functions which satisfy the infinite ODE....

The construction \eqref{eq:homogenODE} 
and the matrix \eqref{eq:Adef} resemble 
in some ways the 
technique called companion linearization used for 
polynomial eigenvalue problems; see e.g. \cite{Mackey:2006:VECT,Tisseur:2001:QUADRATIC}.  
The algorithm known as the infinite Arnoldi method 
\cite{Jarlebring:2012:INFARNOLDI} is an algorithm for nonlinear eigenvalue
problems (not necessarily polynomial). One
variant of the infinite Arnoldi method can be interpreted as
the Arnoldi method \cite{Saad:2011:EIGBOOK} applied to the 
companion linearization
of a truncated Taylor expansion. Due to a particular structure of the
companion matrix, the infinite Arnoldi method  is also equivalent to
the application of the Arnoldi method on an infinite-dimensional companion matrix. This equivalence is
consistent with the observation that many
attractive features of the Arnoldi method appear
to be present also in the infinite Arnoldi method.

We will in this paper illustrate that the 
underlying techniques used to derive the
infinite Arnoldi method
can also be applied to \eqref{eq:semilinearODE}. Similar to the 
infinite Arnoldi method, the presented algorithm
can be interpreted as an exponential integrator applied to 
a truncated problem, as well as the integrator applied to 
an infinite-dimensional problem. 
An important feature of this construction is that
the algorithm does not require a choice of a truncation parameter in 
the expansion \eqref{eq:expansion}, making it in a sense applicable
to arbitrary nonlinearities.

The paper is structured as follows. The infinite-dimensional
properties of the algorithm are derived in Section~\ref{sect:infreform}. 
Although the construction in Section~\ref{sect:infreform}
is general for essentially arbitrary basis, the convergence proofs
are basis dependent. We show that
the algorithm converges for several bases (scaled monomials, Bessel functions
and modified Bessel functions) under certain conditions, 
i.e., the truncation of \eqref{eq:expansion} converges and the
derivatives $g^{(\ell)}(0)$ of the nonlinearity are bounded with respect to the
linear operator $A$ in a certain way.
This convergence theory is presented  in Section~\ref{sect:convergence}.
We illustrate the properties  of the algorithm and its variants 
in Section~\ref{sect:examples} including
comparisons with other algorithms. %We compare it

We will mostly use standard notation.
$(H_N)_{i,j}$ denotes the element at the $i$th row and $j$th column of $H_N$.
Analogously, the colon notation  will be used to denote
entire rows and columns, e.g., $V_{k,:}$ corresponds
to the vector in the $k$th row  of the matrix 
$V$. 
We will also extensively use infinite-dimensional matrices. 
More precisely, we will work with sequences of 
matrices $W_N\in\RR^{n\times N}$, $N=0,\ldots$, which are
nested, i.e., $W_{N-1}\in\RR^{n\times(N-1)}$ are the first $N-1$ columns
of $W_{N}$, and $W_\infty$ will be the corresponding infinite-dimensional matrix. We will also consider
sequences of square matrices $H_N\in\RR^{N\times N}$, where $H_{N-1}$ is the leading
submatrix of $H_N$.  The infinite-dimensional operator associated with the limit
will be denoted $H_\infty\in\RR^{\infty\times\infty}$.  We will use $e_i$ 
to denote the $i$th unit vector of consistent size.
Throughout the paper, $\|\cdot\|$ denotes the Euclidean vector norm or the spectral matrix norm,
unless otherwise stated.

%in particular with a two-step approach consisting of solving a
% linear homogeneous problem and a 

%* relate to infinite Arnold paper and state general idea * 

%* Relate to \cite{KoskelaOstermann} *
%A simple example is given by the
%Taylor expansion of $g(t)$ around $t=0$. In this case
%$\phi_\ell(t) = t^\ell/\ell!$ and $w_\ell = g^{(\ell)}(0)$. Integrators based on
%the Taylor expansion of $H$ were considered in \cite{KoskelaOstermann}. \\
%

%The rest of the article is structured as follows...
%In Section ? we give more examples of linearizations of the form ?
%and show the convergence analysis.
%
%\newpage

\section{Preliminaries} \label{sect:preliminaries}
\subsection{Infinite-dimensional reformulation}\label{sect:infreform}
%In this Section we describe the main ideas behind the iterative methods we propose
%for the time integration of \eqref{eq:semilinearODE}.
%To this end, we first reformulate the ODE \eqref{eq:semilinearODE} as a linear
%infinite-dimensional system. This is stated by the following lemma
%which follows directly from the expansion \eqref{eq:expansion} and the semilinear
%structure of the ODE \eqref{eq:semilinearODE}.
%
%The basis of the  algorithm in this paper is
%a transformation of the inhomogeneous ODE 
At first we will show that the inhomogeneous ODE 
\eqref{eq:semilinearODE} is equivalent to an infinite-dimensional homogeneous ODE.
The reformulation is illustrated
in the following lemma and can be interpreted
as an analogous transformation illustrated 
for monomials and truncated Taylor expansion in 
\eqref{eq:homogenODE} and \eqref{eq:Adef}, but
without truncation and for arbitrary basis functions satisfying \eqref{eq:inf_system}.

\begin{lem}[Infinite-dimensional reformulation]\label{lem:infreform}
Consider the initial value problem \eqref{eq:semilinearODE},
and a sequence of basis functions $[\phi_i]_{i=0}^\infty$
which satisfy \eqref{eq:inf_system}. 
Moreover, suppose that the function $g$ in \eqref{eq:semilinearODE} 
can be expanded as \eqref{eq:expansion}, and let $W_\infty=[w_0,w_1,w_2,\ldots] \in \CC^{n \times \infty}$ 
denote the expansion coefficients.
% $W_k = [w_0,\ldots,w_{k-1}] \in \CC^{n\times k}$, and
% \[
%  v_k = \begin{bmatrix} u \\ v_1 \\ \vdots \\ v_k \end{bmatrix} \in \CC^{n+k}.
% \]
% Furthermore, suppose (?) holds for the functions $y_\ell$.
  \begin{itemize}
    \item[(a)] Suppose $u(t)$ is a solution to \eqref{eq:semilinearODE}. Then
\begin{equation}\label{eq:infODE0}
\frac{d}{dt}
\begin{bmatrix}
u\\
\phi_0\\
\vdots
\end{bmatrix}=
\begin{bmatrix}
A &W_\infty\\
0  & H_\infty
\end{bmatrix}
\begin{bmatrix}
u\\
\phi_0\\
\vdots
\end{bmatrix}, \quad \begin{bmatrix}
u(0)\\
\phi_0(0)\\
\vdots
\end{bmatrix} =
\begin{bmatrix} u_0 \\ e_1 \end{bmatrix}.
\end{equation}
\item[(b)] 
Suppose $v(t)$ satisfies 
\begin{equation}\label{eq:infODE}
v'(t)=
\begin{bmatrix}
A &W_\infty\\
 0 & H_\infty
\end{bmatrix}
v(t), \quad v(0) = \begin{bmatrix} u_0 \\ e_1 \end{bmatrix}.
\end{equation}
%where $\begin{bmatrix} v_{n+1}(t) & v_{n+2}(t) & \ldots \end{bmatrix}^{\mathrm{T}}$ gives
%the unique solution of the ODE \eqref{eq:inf_system}. Then,
Then the function $u(t):= \begin{bmatrix} I_n & 0 \end{bmatrix} v(t)$ is
the unique solution to \eqref{eq:semilinearODE}.
\end{itemize}
\end{lem}
\begin{proof}
The equation \eqref{eq:infODE0} is easily verified
by considering the individual blocks. The first $n$ rows of \eqref{eq:infODE0} satisfy 
$u'(t)=Au(t)+\sum_{\ell=0}^\infty W_\ell\phi_\ell(t)=Au(t)+g(t)$. Rows $n+1, n+2,\ldots$ are precisely the conditions in
\eqref{eq:inf_system1}.  In order
to show \eqref{eq:infODE}, first note that the rows 
$n+1,n+2,\ldots$ in \eqref{eq:infODE} reduce to the equation
\[
\frac{d}{dt}\begin{bmatrix}
v_{n+1}(t)\\
\vdots\end{bmatrix}=
H_\infty 
\begin{bmatrix}
v_{n+1}(t)\\
\vdots\end{bmatrix},\;\;
\begin{bmatrix}
v_{n+1}(0)\\
\vdots\end{bmatrix} =e_1.
\]
Since the operator $H_\infty$ has a finite norm by assumption \eqref{eq:HNasm}, 
it follows from the Picard-Lindel\"of theorem that there exists a unique solution. 
This solution is the sequence of basis functions $\phi_0,\phi_1,\ldots$, since they satisfy 
this ODE by assumption, i.e.,  $v_{n+1+i}=\phi_i$ for all $i\in\NN$. The conclusion follows 
by substituting $v(t)$ into the first $n$ rows
in \eqref{eq:infODE}.
\end{proof}
%Our numerical strategy is to use the formulation \eqref{eq:infODE} to construct iterative methods
%for the initial value problem \eqref{eq:semilinearODE}. This will be introduced in the following section.
%In order to motivate this general framework, we give two more examples
%of expansions of the form \eqref{eq:expansion}.
%
\vspace{5mm}

\subsection{Characterization of basis functions $\phi_\ell$}\label{sect:otherexps}

As mentioned in the introduction (in particular in formula \eqref{eq:H_jordan}),
it is straightforward to verify that the scaled monomials
satisfy the condition \eqref{eq:inf_system} required 
for the basis functions. Although the algorithm
described in the following
section applies for any basis functions satisfying \eqref{eq:inf_system}, 
we concentrate the discussion on  specialized results for two additional types of functions. 
We will now show that the Bessel functions and the modified 
Bessel functions of the first kind satisfy \eqref{eq:inf_system}.

The Bessel functions of the first kind are defined by 
(see e.g.~\cite{Stegun:1964:HANDBOOK}), 
$J_\ell(t):=\frac{1}{\pi}\int_{0}^\pi\cos(\ell\tau -t\sin(\tau))\,d\tau$,  
for $\ell\in\NN$, and they satisfy 
%Derivatives of these functions satisfy  a recursive 
\begin{subequations}\label{eq:Bessel_properties}
\begin{eqnarray}
J_\ell'(t)&=&\frac12 (J_{\ell-1}(t)-J_{\ell+1}(t)).\\
J_{-\ell}(t)& =& (-1)^\ell J_\ell(t),\;\;\ell>0\\
J_\ell(0)&=&
\begin{cases}
   1 & \textrm{ if } \ell=0\\
   0 & \textrm{ otherwise}\\
\end{cases}.\label{eq:Bessel_symmetry}
\end{eqnarray}
\end{subequations}
Let $\bar J_N(t) = \begin{bmatrix} J_0(t) & J_1(t) & \dots & J_{N-1}(t) \end{bmatrix}^\mathrm{T}   \in\RR^{N}$,
i.e., a vector of Bessel functions
with non-negative index. Moreover, let
\begin{equation} \label{eq:H_bessel}
 H_N = 
\begin{bmatrix} 0 & -1 & &  & \\
	      \tfrac{1}{2} & 0 & -\tfrac{1}{2} &  &  \\
	         & \ddots & \ddots & \ddots &  \\
	         & &\ddots & 0 & -\tfrac{1}{2} &  \\
	         & & &\tfrac{1}{2} &0 & \end{bmatrix}
\in\CC^{N\times N}.
	         \end{equation}
From the  relations \eqref{eq:Bessel_properties}, 
we easily verify that the Bessel functions of the first kind are 
solutions to the infinite-dimensional ODE of the form \eqref{eq:inf_system},
with $H_N$ given by $\eqref{eq:H_bessel}$. More precisely, 
$$
\bar J_\infty'(t) = H_\infty \bar J_\infty(t), \quad \bar J_\infty(0) = e_1.
$$
%More precisely, 
%
%we see that the Bessel functions of the first kind
%satisfy the infinite dimensional initial value problem
%$$
%\bar J_\infty'(t) = H_\infty \bar J_\infty(t), \quad \bar J_\infty(0) = e_1.
%$$
With similar reasoning we can establish an 
ODE \eqref{eq:inf_system} also for  the modified Bessel functions of the first kind, which are defined by
\begin{equation} \label{eq:modified_Bessels}
I_\ell(t) := (- \ii)^n J_n(\ii t).
\end{equation}
and satisfy 
%\label{eq:Bessel_properties2}
$I_\ell'(t)=\frac12 (I_{\ell-1}(t)+I_{\ell+1}(t))$, $\ell \in \mathbb{N}$.
%Derivatives of these functions satisfy the recursion
%\begin{equation} \label{eq:Bessel_recursions2}
%\end{equation}
%and as they also satisfy properties \eqref{eq:Bessel_properties}, they satisfy the infinite ODE 
These properties lead to the infinite-dimensional ODE 
$$
\bar I_\infty'(t) = H_\infty \bar I_\infty(t), \quad I(0) = e_1,
$$
where $\bar I_N(t) = \trans{\begin{bmatrix} I_0(t)  & I_1(t) & \dots & I_{N-1}(t) \end{bmatrix}}$ and 
\begin{equation} \label{eq:H_bessel_modified}
 H_N = 
\begin{bmatrix} 0 & 1 & &  & \\
	      \tfrac{1}{2} & 0 & \tfrac{1}{2} &  &  \\
	         & \ddots & \ddots & \ddots &  \\
	         & &\ddots & 0 & \tfrac{1}{2} &  \\
	         & & &\tfrac{1}{2} &0 & \end{bmatrix}
\in\CC^{N\times N}.
\end{equation}
Therefore, we can show that the Bessel functions and the modified Bessel functions 
of the first kind satisfy
\eqref{eq:inf_system}, with an explicitly given constant $C$.
\begin{lem}[Basis functions]
The conditions for the basis functions 
in \eqref{eq:inf_system} are satisfied with $C=2$ for,
\begin{itemize}
  \item[(a)] scaled monomials, i.e., $\phi_i(t)=t!/i!$, with $H_\infty$ defined by \eqref{eq:H_jordan};
  \item[(b)] Bessel functions, i.e., $\phi_i(t)=J_i(t)$, with $H_\infty$ defined by \eqref{eq:H_bessel}; and 
  \item[(c)] modified Bessel functions, i.e., $\phi_i(t)=I_i(t)$, with $H_\infty$ defined by \eqref{eq:H_bessel_modified}.
\end{itemize}
\end{lem}
\begin{proof}
Statement (a) follows from the definition. 
The conditions \eqref{eq:inf_system1} 
have been shown already for (b) and (c), 
since they  follow directly from \eqref{eq:Bessel_properties}
and \eqref{eq:modified_Bessels}.
It remains to show that the uniform bound
\eqref{eq:HNasm} is satisfied for (b) and (c).
Note that in both cases (b) and (c) we can express $H_N$ as
 $H_N=T_N+E_N$, 
where $E_N=\pm \tfrac12 e_1e_2^\mathrm{T}$ and $T_N$ is a (finite) band Toeplitz matrix,
for any $N=2,\ldots,\infty$. 
We now invoke a general result
for (finite) band Toeplitz matrices \cite[Theorem~1.1]{Boettcher:2000:TOEPLITZ}
which implies that 
$\|T_N\|\le \|T_\infty\|=1$. Hence, $\|H_N\|\le\|T_N\|+\|E_N\|=3/2$,
and \eqref{eq:HNasm} holds with $C=2$. 
%Moreover, $\|T_\infty\|=1$ also from \cite[Theorem~1.1]{Boettcher:2000:TOEPLITZ}.
\end{proof}

\subsection{Characterization of expansion coefficients $w_\ell$}

In principle, the algorithm will be applicable
to any problem for which there is a convergent expansion of the form \eqref{eq:expansion} 
with some coefficients $w_\ell\in\mathbb{C}^n$, $\ell\in\NN$. In practice
these coefficients may not be explicitly available. We will now 
characterize a relationship between the coefficients
and the derivatives of $g$. This will be necessary in the 
theoretical convergence analysis in Section~\ref{sect:convergence}
and also useful in numerical evaluation of the coefficients (Section~\ref{sect:examples}).

Assume that an expansion of the form \eqref{eq:expansion} exists and let  
$$
W_N = \begin{bmatrix} w_0 & w_1 & \ldots & w_{N-1} \end{bmatrix}.
$$
By considering the $l$th derivative 
of $g(t)$  and
using the properties of basis functions \eqref{eq:inf_system} 
we have that
$$
g^{(\ell)}(0) = W_\infty H_\infty^\ell e_1=
W_NH_N^\ell e_1 \quad \textrm{for all} \quad \ell< N.
$$
In the last equality we used the fact that $H_\infty$ is a Hessenberg matrix, and
that all elements of $H_\infty^\ell e_1$ except the first $\ell+1$ elements will be zero. 
The non-zero elements will also be equal to $H_N^\ell e_1$.
We now define the upper-triangular matrix 
\begin{equation} \label{def:krylov_matrix}
K_N(H_N,e_1) = \begin{bmatrix} e_1 & H_N e_1 & \ldots & H_N^{N-1} e_1 \end{bmatrix},
\end{equation}
and the matrix $G_N$ as
\begin{equation} \label{def:G_N}
G_N = \begin{bmatrix} g(0) & g'(0) & \ldots & g^{(N-1)}(0) \end{bmatrix}.
\end{equation}
From the definition it follows that
\begin{equation} \label{eq:Bessel_Krylov_relation}
W_N= G_N  K_N(H_N,e_1)^{-1}    \quad \textrm{for all} \quad N\geq 1,
\end{equation}
under the condition that $K_N(H_N,e_1)$ is invertible. 
In a generic situation, the relation \eqref{eq:Bessel_Krylov_relation} 
can be directly used to  compute the coefficients $w_\ell$, $\ell\in\NN$, given the derivatives of $g(t)$. 
For the Bessel functions and the modified Bessel functions of the first kind, we can characterize
the coefficients with a more explicit (and more numerically robust) formula
 involving 
%As the following lemma shows, the coefficients of the matrix $K_N(H_N,e_1)^{-1}$ come from
the monomial coefficients of the Chebyshev polynomials of the first kind.
\begin{lem} \label{lem:krylov_matrix_inverse}
Let  $T_{k,\ell}$ be the monomial coefficients
of the $k$th Chebyshev polynomial, i.e., 
$T_k(x)=\sum_{\ell=0}^kT_{k,\ell}x^\ell$.
\begin{itemize}
  \item[(a)] For scaled monomials, i.e., $\phi_k(t)=t!/k!$, the expansion 
coefficients are given by $w_k=g^{(k)}(0)$, for $k\in\NN$.
   \item[(v)] %Let $H_N \in \RR^{N \times N}$ be defined as in \eqref{eq:bessel_H_N}
For the Bessel functions of the first kind, i.e., $\phi_\ell(t)=J_\ell(t)$,
% in expansion \eqref{eq:expansion} are the Bessel functions,
the expansion coefficients $w_\ell$ are given by,
$$
w_0 = g(0), \quad w_k = 2 \sum\limits_{\ell=0}^k (-1)^\ell \, T_{k,\ell} \, g^{(\ell)}(0), \;\;k=1,\ldots.
$$
  \item[(c)]
For the modified Bessel functions of the first kind, i.e., $\phi_\ell(t)=I_\ell(t)$,
% in expansion \eqref{eq:expansion} are the Bessel functions,
the expansion coefficients $w_\ell$ are given by,
%
% For modified  $\phi_\ell$ in expansion \eqref{eq:expansion} are the modified Bessel functions,
%the coefficients $w_\ell$ are given by
$$
w_0 = g(0), \quad w_k = 2 \sum\limits_{\ell=0}^k T_{k,\ell} \, g^{(\ell)}(0), \;\;k=1,\ldots.
$$
\end{itemize}
\end{lem}
\begin{proof}
Case (a) follows from the definition.
Consider case (c) with the modified Bessel functions, i.e., let $H_N$ be
given by \eqref{eq:H_bessel_modified}. 
The proof is based on showing that 
\begin{equation}
   K_N(H_N,e_1)^{-1}=
   2\begin{bmatrix}
     \frac12 T_{0,0}& T_{1,0} & \cdots  & T_{N-1,0}    \\
     0      & T_{1,1}  &  &T_{N-1,1} \\
     \vdots      & &  \ddots  & \vdots  \\
     0 & & \cdots    & T_{N-1,N-1}
   \end{bmatrix},\label{eq:Kinv_modbessel}
\end{equation}
from which the conclusion follows directly from \eqref{eq:Bessel_Krylov_relation} and the fact that 
$T_{k,k}=2^{k-1}$ for any $k>0$. 
%The element $(K_N(H_N,e_1)^{-1})_{1,1}$ is easily 
%verified (and is a break of symmetry).

We will first prove \eqref{eq:Kinv_modbessel} for columns $k=2,3,\ldots,N$.
From \eqref{def:krylov_matrix} and 
\eqref{eq:H_bessel_modified}  
we directly identify 
 that 
\[
K_{N+1}(H_{N+1},e_1)=
\begin{bmatrix} K_N(H_N,e_1) & H_N^Ne_1 \\ 0 & 2^{-N}
                        \end{bmatrix}.
\]
%has diagonal elements explicitly given as
%$\left( K_N(H_N,e_1) \right)_{i,i} = \frac{1}{2^{i-1}}$. 
Moreover, by explicitly formulating 
the Schur complement \cite[Section~3.2.11]{Golub:2013:MATRIX} we have that
%  $K_{N+1}(H_{N+1},e_1)$, we have 
\begin{equation}
   K_{N+1}(H_{N+1},e_1)^{-1}= 
\begin{bmatrix} K_N(H_N,e_1)^{-1} & -2^NK_N(H_N,e_1)^{-1}H_N^Ne_1 \\ 0 & 2^N
                        \end{bmatrix}.\label{eq:Kinv_modbessel2}
\end{equation}
%In order to show the equality of the last column of \eqref{eq:Kinv_modbessel}, 
%we 
Now let $p_N(\lambda)$ be the characteristic polynomial of $H_N$, i.e.,
$p_N(\lambda) = \operatorname{det}(\lambda I - H_N)$.

By  expanding the determinant of $\lambda I -H_N$   
for the last row, we find that
\begin{equation}\label{eq:rec_H}
p_N(\lambda) = \lambda p_{N-1}(\lambda) - \frac{1}{4} p_{N-2}(\lambda).
\end{equation}
Now let 
$\wt p_N(\lambda) := 2^{N-1} p_N(\lambda)$ %= \sum\limits_{i=0}^N \alpha_i \lambda^i.
which satisfies the recursion $\wt p_N(\lambda) = 2 \lambda \wt p_{N-1} (\lambda) - \wt p_{N-2} (\lambda)$. This is exactly the recursion of the Chebyshev polynomials. 
We have  $\wt p_1(\lambda) = \lambda = T_1(\lambda)$
and $\wt p_2(\lambda) = ( 2\lambda^2 -1) =  T_2(\lambda)$.
Hence, by induction starting with $N=1$ and $N=2$ it follows that
$\wt p_N(\lambda) = T_N(\lambda), \quad \textrm{for all} \quad N\geq 1$. 
Note that  $\wt p_0\neq T_0$. The Cayley-Hamilton theorem implies that $0=p_N(H_N)=\wt p_N(H_N)=T_N(H_N)$
and in particular $0=2\wt p(H_N)e_1$, i.e., 
\begin{equation} \label{eq:K_N_2}
-2^N H_N^N e_1 = \sum\limits_{i=0}^{N-1} 2T_{N,i} H_N^i e_1.
\end{equation}
The first $N$ rows of the last column  of \eqref{eq:Kinv_modbessel2} can now be
expressed as
\[
 -2^NK_N(H_N,e_1)^{-1}H_N^Ne_1=
 2K_N(H_N,e_1)^{-1}\left(\sum\limits_{i=0}^{N-1} T_{N,i} H_N^i e_1\right)=  
 2\begin{bmatrix}
   T_{N,0}\\
   \vdots\\
   T_{N,N-1}
 \end{bmatrix}
\]
The structure in  \eqref{eq:Kinv_modbessel}
for columns $k=2,\ldots,N$ 
follows by induction. The first column can be verfied directly 
by noting that $K_1(H_1,e_1)=1=T_{0,0}$.

The proof for the case (b) goes analogously. From \eqref{eq:H_bessel_modified} we see that
in this case the characteristic polynomial $p_N(\lambda)$ of $H_N$ satisfies the recursion
\begin{equation} \label{eq:rec_H_mod}
p_N(\lambda) = \lambda p_{N-1}(\lambda) + \frac{1}{4}p_{N-2}(\lambda).
\end{equation}
Defining $\wt p_N(\lambda) = 2^{N-1} p_N(\lambda)$ and writing out the recursions we find 
(similarly to the case (c)) that $\wt p_N(\lambda) = (- \ii)^N T_N(\ii \lambda)$.
Comparing \eqref{eq:rec_H} and \eqref{eq:rec_H_mod} we see that $\wt p_N(\lambda)$ is of the form 
$\wt p_N(\lambda) = \sum_{\ell=0}^N \abs{T_{N,\ell}} \lambda^\ell$. The claim follows from this.
\end{proof}

\begin{remark}[Combining with formulas for the Chebyshev polynomials]\rm
The coefficients $T_{k,\ell}$ are given by the explicit expression 
(see~\cite[pp.\;775]{Stegun:1964:HANDBOOK})
\begin{equation} \label{eq:Cheby_explicit}
T_k(x)  =  \sum\limits_{\ell=0}^{\left \lfloor \frac{k}{2} \right \rfloor} (-1)^\ell \frac{k(k-\ell-1)!}{\ell ! (k-2\ell)!}  2^{k-2\ell - 1} 
        x^{k-2\ell}.
\end{equation}
Thus, when the basis functions $\phi_\ell$ are the modified Bessel functions of the first kind,
the coefficients of the expansion \eqref{eq:expansion} are explicitly given by
\begin{equation}  \label{eq:Neumann}
\begin{aligned} 
w_k &=  \sum\limits_{\ell=0}^{\left \lfloor \frac{k}{2} \right \rfloor} (-1)^\ell \frac{k(k-\ell-1)!}{\ell ! (k-2\ell)!}  2^{k-2\ell - 1} 
       g^{(k-2\ell)}  \\
&= \sum\limits_{\ell=0}^{\left \lfloor \frac{k}{2} \right \rfloor} (-1)^\ell \frac{k(k-\ell-1)!}{\ell ! }  2^{k-2\ell - 1} 
         \left( \frac{1}{2\pi\ii} \int_\Gamma \frac{g(\lambda)}{\lambda^{k - 2 \ell}} \, \dd \lambda \right),
\end{aligned}
\end{equation}
where in the last step we have used the Cauchy integral formula. For the expansions with the Bessel functions
of the first kind we get exactly the same formula, with $(-1)^\ell$ replaced by 1 in the summand, which is also
given in~\cite[Sec.\;9.1]{Watson}.
\end{remark}

%* mention that what we do is actually a Neumann-series and we results
%for Neumann series that appear not available in the literature *

\section{Infinite Arnoldi exponential integrator for (\ref{eq:semilinearODE})}

Consider for the moment a linear (finite-dimensional) homogeneous ODE
\begin{equation}
  y'(t)=B y(t),\;\; y(0)=b\label{eq:ODEhomogeneous}
 \end{equation}
where $y(t)\in\CC^{n}$, with the solution
given by the matrix exponential $y(t)=\exp(tB)b$.
%\begin{equation}\label{eq:exptB}
 % y (t)=\exp(tB)b.
%\end{equation}
Algorithms for \eqref{eq:ODEhomogeneous} based on Krylov methods 
are typically constructed as follows, see \cite{Hochbruck:1997:KRYLOV}
and references therein for further details. 
By carrying out $N$ steps of 
the Arnoldi process for $B$ and $b$ we obtain
the Hessenberg matrix $F_N$ and the orthonormal matrix $Q_{N+1}\in\CC^{n\times (N+1)}$ 
that satisfy the so called Arnoldi relation
\begin{equation}\label{eq:arnoldirelation}
  BQ_N=Q_NF_N+f_{N+1,N}q_{N+1}e_N^\mathrm{T},
\end{equation}
where $q_i$ denotes the $i$th column of $Q_{N+1}=[q_1,\ldots,q_{N+1}]=[Q_N,q_{N+1}]$ 
and $f_{i,j}$ the $i,j$ element of $F_{N}$, and $q_1=b/\beta$ with $\beta:=\|b\|$. 
The columns of $Q_{N}$ form 
an orthogonal basis of the Krylov subspace
\[
 \mathcal{K}_N(B,b) = \operatorname{span}(b,B b,\ldots,B^{N-1}b). % = \operatorname{span}(q_1,\ldots,q_N)
\]
As a consequence of \eqref{eq:arnoldirelation}, the Hessenberg
matrix $F_N$ is the projection of $B$ onto the Krylov subspace $\mathcal{K}_N(B,b)$, i.e., 
$F_N=Q_N^*BQ_N$.

The Krylov approximation of \eqref{eq:ODEhomogeneous} is subsequently given by 
\begin{equation}\label{eq:krylov_approx}
  y (t) =  \exp(tB)b\approx Q_N\exp(tF_N)e_1\beta.
\end{equation}
Krylov approximations of the matrix exponential has for instance 
been used in \cite{Gallopoulos:1992:EFFICIENT,Park:1986:UNITARY,Vorst:1987:MATFUN}.

The first justification of the proposed algorithm is based on applying a Krylov approximation analogous
 to \eqref{eq:krylov_approx}
for the infinite-dimensional homogeneous ODE given in Lemma~\ref{lem:infreform}.
Although this construction is infinite-dimensional, 
it turns out that due to the  structure of $A_\infty$ and the starting vector
$b = [u_0^\mathrm{T},e_1^\mathrm{T}]^\mathrm{T}\in\CC^\infty$, the 
basis matrix $Q_N$ has a particular structure which can be exploited.

\begin{lem}[Basis matrix structure]\label{thm:Vmstruct}
Let $Q_N\in\CC^{\infty\times N}$ be the matrix
generated by the Arnoldi method applied to the infinite 
matrix $A_\infty$ given by \eqref{eq:Adef} and the starting vector 
$b=[u_0^\mathrm{T},e_1^\mathrm{T}]^\mathrm{T}\in\CC^\infty$. Let $q_{1,j}\in\CC^{n+1}$, 
for $j=1\ldots N$, be the first $n+1$ rows of $Q_N$ and let $q_{i,j}\in\CC$, $i=2,\ldots,j=2,\ldots,N$
correspond to the rows $n+2$, $n+3$, $\ldots$. 
Then, the basis matrix $Q_N$ has the block-triangular structure
\begin{equation}\label{eq:Vmstruct}
Q_N=\begin{bmatrix}
q_{1,1}     & q_{1,2}  & \dots   & q_{1,N}\\
       0   & q_{2,2}  & \dots & q_{2,N}\\
       \vdots   &0    &\ddots   &\vdots\\
            & \vdots       &  0      &q_{N,N}\\
            &         &  \vdots     &0\\
            &        &       &\vdots\\
\end{bmatrix}\in\CC^{\infty\times N}
\end{equation}
\end{lem}
\begin{proof}
The proof can be done by induction. For $N=1$ the statement
is trivial. If we assume $Q_N$ has the structure
\eqref{eq:Vmstruct}, at step $N$ the Arnoldi method
will generate a new vector $q_{:,N+1}\in\CC^{\infty}$ which 
is a linear combination of $A_\infty q_{:,N}$ and the 
columns of $Q_N$. Due to the fact that
$H_\infty$ is a Hessenberg matrix, and
$A_\infty$ has the structure \eqref{eq:Adef},
$A_\infty q_{:,N}$
will have one more non-zero element 
than $q_{:,N}$. This completes the proof.
\end{proof}

The zero-structure in the basis matrix $Q_N$ 
revealed in Lemma~\ref{thm:Vmstruct},
suggests that we can implement the
Arnoldi method for \eqref{eq:infODE} 
by only storing the non-zero part of $Q_N$. 
By noting that the orthogonalization
also preserves the basis matrix structure,
we can derive an algorithm where in every step the basis
matrix is expanded by a column and a row.
We note that the infinite Arnoldi
method for nonlinear eigenvalue problems has
a similar property \cite[Section~5.1]{Jarlebring:2012:INFARNOLDI}.
The proposed algorithm is specified in 
Algorithm~\ref{alg:infarn}. As 
is common for the Arnoldi method, in Step~\ref{step:reorth}
we used reorthogonalization 
if necessary. 
%By the notation 
%The notation $\underline{B}$ denotes
%a matrix where a zero row is appended to the matrix $B$.
%
%
%* brief desc of Krylov methods for exponential *
%
%\[
%   \tilde{v}(t)\approx V\exp(tH)e_1\beta
%\]
%where $V$ is generated with Arnoldi's method initiated with $v_0/\|v_0\|$.
%
%* apply this approach to \eqref{eq:infODE} *

%* compact storage *

\begin{algorithm}[t]%\SetLine
\caption{The infinite Arnoldi exponential integrator
for \eqref{eq:semilinearODE} \label{alg:infarn}}%
\SetKwInOut{Input}{Input}\SetKwInOut{Output}{output}
\Input{$u_0\in\CC^n$, $t\in\RR$, $w_0,w_1\ldots \in\CC^n$}
\Output{The approximation $u_N^{IA}\approx u(t)$}
\BlankLine
\nl Let $\beta=\|u_0\|$, $Q_1=u_0/\beta$,
  $\underline{\wt F}_0=$empty matrix\\
\For{$k=1,2,\ldots,N$}{
\nl Let $q_{k}=Q(:,k)\in\CC^{n+k-1}$\\
\nl Compute $w :=A_kq_k$\\
\nl Let $\underline Q_k$ be  $Q_k$ with one zero row added\\
\nl Compute $h=\underline Q_k^*w$\label{step:orth:h}\\
\nl Compute $w_\perp:=w- \underline Q_kh$\label{step:orth:wperp}\\
\nl Repeat Step~\ref{step:orth:h}-\ref{step:orth:wperp}
if necessary\label{step:reorth}\\ 
\nl Compute $\alpha=\|w_\perp\|$ \\
\nl Let $\underline{F}_k=\begin{bmatrix}
\underline{F}_{k-1} &h   \\
 0 &\alpha \\
\end{bmatrix}$\\
\nl Let $Q_{k+1}:=[\underline Q_k,w_\perp/\alpha]$
}
\nl Let $F_N\in\RR^{N\times N}$ be the leading submatrix
of $\underline{F}_N\in\RR^{(N+1)\times N}$\\
\nl Compute the approximation $u_N^{IA}=
\begin{bmatrix} I_n & 0 \end{bmatrix} Q_N \exp(t F_N) e_1 \beta$
\end{algorithm}

%\begin{algorithm}\label{alg:infarn}
%\begin{algorithmic}
%\STATE TODO
%\STATE $u_k^{IA}=$...
%\end{algorithmic}
%\end{algorithm}
%
%* introduce approx $A_N$ * 

%* describe Krylov approx of $A_N$ after $m$ steps 
%yielding $u_{m,N}$ * 

Another natural procedure to compute 
a solution to \eqref{eq:semilinearODE}
would be to truncate the matrix 
$H_\infty$ and thereby $A_\infty$
%and modify the basis functions 
such that we obtain a linear finite-dimensional ODE
\[
\tilde{v}'(t)=A_m\tilde{v}(t),\;\; \tilde{v}(0)=\begin{bmatrix}
u_0\\
e_1
\end{bmatrix}
\]
using \eqref{eq:homogenODE} and subsequently applying
the standard Krylov approximation \eqref{eq:krylov_approx}
on this finite-dimensional ODE.
It turns out that this approach will provide an algorithm 
equivalent to Algorithm~\ref{alg:infarn},
if the truncation parameter is chosen
larger or equal to the number of Arnoldi steps.
Hence, in addition to the fact that Algorithm~\ref{alg:infarn}
can be interpreted as an infinite-dimensional
Krylov approximation of \eqref{eq:infODE}, the algorithm
is also equivalent to the finite-dimensional
Krylov approximation corresponding to the truncated matrix, 
if the truncation parameter is chosen larger than the number of steps.

\begin{lem} \label{lem:Arnoldi_property}
%Let $A_k$ and $u_k$ be defined in \eqref{eq:Adef}. 
%Then, the $m$-step Arnoldi approximation same
Consider $N$ steps of the Arnoldi method
 applied to $A_m\in\CC^{(n+m)\times (n+m)}$
with starting vector $b=[u_0^\mathrm{T},e_1^\mathrm{T}]^\mathrm{T}\in\CC^{n+m}$.
Let $u_{N,m}$ be the corresponding
Krylov approximation, i.e., $u_{N,m}:=\begin{bmatrix} I_n & 0 \end{bmatrix} Q_N \exp(t F_N) e_1 \beta$. 
Then, for any $m \geq N$, we have $u_N^{IA}=u_{N,m}$.
%Growing structure...
\end{lem}
\begin{proof}
This follows directly from the zero-structure 
of the basis matrix in \eqref{eq:Vmstruct}, which holds
also for finite $m$, when $m\ge N$.%
\end{proof}

\section{Convergence analysis}\label{sect:convergence}
We saw in Lemma~\ref{lem:Arnoldi_property} that 
although the Algorithm~\ref{alg:infarn} is derived from  Arnoldi's method applied to an 
infinite-dimensional operator $A_\infty$, the result of $N$ steps
of the algorithm can also be interpreted as Arnoldi's method applied to 
the truncated matrix $A_m$ for any $m \le N$. 
In order to study the convergence we will set $m=N$ and use the exact solution associated
with the truncated matrix $A_N$, denoted by $u_N(t)$. More precisely,
\begin{equation} \label{eq:N_exp}
 u_N(t) := \begin{bmatrix} I_n & 0 \end{bmatrix} \exp(t A_N) u_N.
\end{equation}
By  trivial subtraction and triangle
inequality we have that the error is bounded by
\begin{equation} \label{eq:error_splitting}
\norm{u_N^{IA}-u(t)} \le \norm{u(t)-u_N(t)} + \norm{u_N(t)- u_N^{IA}}.
\end{equation}
The first term $\norm{u(t)-u_N(t)}$ can be interpreted
as an error associated with $A_N$ (the truncation of $A_\infty$) and
is not related to Arnoldi's method,
whereas the second term $\norm{u_N(t)- u_N^{IA}}=\norm{u_N(t)- u_{N,N}}$
can be seen as an error associated with the Arnoldi approximation
of the matrix exponential.
The following two subsections are devoted to the characterization of these
two errors.

%
%After $N$ iterations of the infinite Arnoldi algorithm we obtain the approximation
%\[
% u_N(t) \approx \begin{bmatrix} I_n & 0 \end{bmatrix} Q_N \exp(t F_N) e_1 \beta=:u^{IA}_N,
%\]
%where $Q_{N+1} \in \CC^{(n+N) \times (N+1)}$ is orthonormal matrix and
%$F_N \in \CC^{N \times N}$ is the Hessenberg matrix given by the Arnoldi
%process applied to the Krylov subspace
%\[
% \mathcal{K}_N(A_N,u_N).
%\]
%
%As shown by Lemma 3, running the infinite Arnoldi algorithm for the semilinear equation \eqref{eq:semilinearODE}
%is equivalent to performing $N$-step Arnoldi approximation for
%\begin{equation} \label{eq:N_exp}
% u_N(t) = \begin{bmatrix} I_n & 0 \end{bmatrix} \exp(t A_N) u_N,
%\end{equation}
%where $A_N \in \CC^{(n+N) \times (n+N)}$ is defined as in \eqref{eq:k_augmented}.
%using the Arnoldi process. 
%
%...Therefore, we split the error given by the infinite Arnoldi iteration as follows

\subsection{Bound for the truncation error}\label{sect:truncationerror}
%We will now see that under general conditions
%the first term 
%in \eqref{eq:error_splitting} vanishes as $N\rightarrow 0$.
%The reasoning leading up to the main theorem (Theorem~XX) 
%is based on separate properties for the monomials
%and bessel functions and presented in Theorem~\ref{thm:convergence_taylor}
%and Theorem~\ref{thm:convergence_bessel}.
%The truncation error (second term in \eqref{eq:}) will
%be analyzed by   
It will turn out that the truncation error 
(first term in \eqref{eq:error_splitting})
can be analyzed by relating it to $\exp(tH_N)e_1$,
i.e.,
%The exact solution of the problem associated with the
%truncated matrix $A_N$ is closely related
%to 
a vector of 
functions generated by the truncated Hessenberg matrix $H_N$.
%Consider the  the vector of functions $\exp(H_Nt)e_1$, i.e.,
%the vector of functions 
%generated by the truncated Hessenberg matrix.
Let $\bar\veps_N$ denote the difference between the basis
functions and the functions generated by the Hessenberg
matrix, 
\begin{equation}
\bar\veps_{N}(t):=
\bar\phi_N(t)-\exp(tH_N)e_1,\;\;\;
\bar\phi_N(t):=\begin{pmatrix}
  \phi_0(t)\\
  \vdots\\
  \phi_{N-1}(t)
 \end{pmatrix}.\label{eq:vepsdef}
\end{equation}
%
%We will see below that 
%if the basis functions are scaled monomials (and $H_\infty$ is
%given by \eqref{eq:H_jordan}) the truncated Hessenberg matrix generates
%exactly the basis functions, such that
% $\bar\veps$ 
%vanishes identically. This does not occur
%for other basis functions and further analysis will be required.
%The first term in \eqref{eq:error_splitting}
%will now be explicitly bounded by using the variation-of-constants formula 
%as follows. 
The following lemma shows that
a sufficient condition for the convergence
of the first term in \eqref{eq:error_splitting} is that 
$\|W_N\bar\veps_N(s)\|\rightarrow 0$.
The following subsections are devoted
to the analysis of  $\|W_N\bar\veps_N(s)\|$
for different basis functions,
and in particular lead up the convergence of the truncation error
given under general conditions in Theorem~\ref{thm:convergence_taylor} and 
Theorem~\ref{thm:convergence_bessel}.

\begin{lem} \label{lem:truncation_error}
Let $u$ be the solution to the ODE  \eqref{eq:semilinearODE} 
and $u_N$ be defined as in \eqref{eq:N_exp}.
Suppose the expansion \eqref{eq:expansion} is uniformly convergent
with respect to $s$.
Then, 
\begin{multline}
 \norm{u(t) - u_N(t)} 
\leq \int\limits_0^t \norm{\ee^{(t-s)A}} \norm{g(s) - W_N\ee^{sH_N}e_1} \, \dd s\leq\\
\int\limits_0^t \norm{\ee^{(t-s)A}}\,\dd s
\left(\max_{s\in[0,t]}\norm{g(s) - W_N\bar\phi_N(s)}+\max_{s\in[0,t]}\norm{W_N\bar\veps_N(s)}\right).\label{eq:aftervarconst}
\end{multline}
Moreover, for every $s\le t$ we have 
\begin{equation}
\norm{g(s) - W_N\bar\phi_N(s)}\rightarrow 0\;\; \textrm{ as }N\rightarrow\infty.
\label{eq:limitexpr}
\end{equation}
\end{lem}
\begin{proof}
The first  bound in \eqref{eq:aftervarconst} follows from the variation-of-constants formula $u(t) = \ee^{tA} u_0 + \int\limits_0^t \ee^{(t-s)A} g(s) \, \dd s$, 
which gives the exact solution for the ODE \eqref{eq:semilinearODE}, and from the representation~\cite[pp.\;248]{Higham}
\begin{align*}
  u_N = \begin{bmatrix} I_n & 0 \end{bmatrix} \ee^{t A_N} u_N &= %\exp \left( \begin{bmatrix} A & W_N \\ 0 & H_N \end{bmatrix} \right) u_N  = 
\begin{bmatrix} I_n & 0 \end{bmatrix} \begin{bmatrix} \ee^{tA} & \int\limits_0^t \ee^{(t-s)A} W_N \ee^{s H_N} \, \dd s \\
0 & \ee^{t H_N} \end{bmatrix} u_N \\ &= \ee^{tA} u_0 + \int\limits_0^t \ee^{(t-s)A} W_N \ee^{s H_N} e_1 \, \dd s.
 \end{align*}
The second inequality follows by adding
and subtracting $\bar\phi_N(s)$.
The limit expression \eqref{eq:limitexpr}
follows from the fact that $g(s) - W_N\bar\phi_N(s)=\sum_{\ell=N}^\infty W_\ell\phi_\ell(s)$, which is
the remainder term in the expansion \eqref{eq:expansion}.
The limit vanishes due to the fact that \eqref{eq:expansion} is uniformly convergent.
\end{proof}

%It remains to show that $W_N\bar\veps_N(s)$ vanishes
%uniformly as $N\rightarrow\infty$, for
%which we need to carry out 
%specialized analysis for different
%basis functions.
%The proof 
%converges to the basis functions $\phi_0,\phi_1,\dots$
%for all $t>0$ as $k\rightarrow \infty$.
%Denoting $\phi^k = \begin{bmatrix} \phi_0 & \hdots & \phi_{k-1} \end{bmatrix}^{\mathrm T}$, this means
%that
%$$
%\norm{\phi^k(t) - \exp(t H_k) e_1 } \rightarrow 0, \quad \textrm{as} \quad k \rightarrow \infty.
%$$ 
%Moreover, we assume that $\forall t>0$
%$$
%\norm{g(t) - \sum\limits_{\ell = 0}^k w_\ell y^k_\ell(t)} \rightarrow 0, \quad \textrm{as} \quad k \rightarrow \infty,
%$$
%and perform the approximation
%\begin{equation} \label{eq:series_approx}
%g(t) \approx \sum\limits_{\ell = 0}^k w_\ell y^k_\ell(t)
%\end{equation}
%

\subsubsection{Bounds on $\|W_N\bar\veps_N\|$: scaled monomial basis}
Suppose $\phi_\ell$ are scaled monomials such that $H_N$ is a transposed Jordan matrix, 
i.e., the truncation
of \eqref{eq:H_jordan}. The definition 
of $\bar\veps_N$ yields $\bar\veps_N^{(k)}(0)=\bar\phi_N^{(k)}(0)-H_N^ke_1$.
It follows from the structure \eqref{eq:H_jordan} 
that, for $k\le N$,  $H_N^ke_1=e_k$  and $H_N^ke_1=0$ if $k>N$. 
Moreover, $\bar\phi_N^{(k)}(0)=e_k\phi_k^{(k)}(0)=e_k$ if $k\le N$ and 
$\bar\phi_N^{(k)}(0)=0$ if $k>N$.
Hence, $\bar\veps_N^{(k)}(0)=0$ for all $k$. Since $\bar\veps_N$ 
is analytic and all derivatives vanish, $\bar\veps_N(t)\equiv 0$. 
%We have shown the following result.
Hence, we have obtained the following result.

\begin{thm} \label{thm:convergence_taylor}
Suppose $\bar\phi_N$ are the scaled monomials,
given by $\phi_\ell(t):=t^\ell/\ell!$, and $H_N\in\RR^{N\times N}$ is the 
leading submatrix of \eqref{eq:H_jordan}. Then, 
\[
  \bar\veps_N(t) \equiv 0
\] 
where $\bar\veps_N(t)$ is given by \eqref{eq:vepsdef}. Consequently,
if the basis functions are the scaled monomials, 
the truncation error $\|u(t)-u_N(t)\|\rightarrow 0$ independent of $t$.
\end{thm}
%Note that for the Taylor case the error satisfies 
%\[
%g(s)-W_Ne^{sH_N}e_1 = \int_0^t \ee^{(t-\tau)A}
%\int_0^\tau \frac{(\tau-\xi)^{p-1}}{(p-1)!}g^{(p)}(t_n+\xi)\,\dd\xi\,\dd\tau.
%\]
%Clearly, 
%$g(s)-W_Ne^{sH_N}e_1\rightarrow 0$ as $N\rightarrow\infty$.
%%\[
% \|g(t)-We^{sH_N}e_1\|
%\]
%
%
%* Remainder term for taylor expansion *
%
%If $\phi_1,\phi_2,\ldots$ are the monomials. Then,
%\[
% \|g(t)-We^{sH_N}e_1\|\le R_N(\xi).
%\]
%Following from \cite{Koskela:2013:EXPONENTIALTAYLOR}, might be useful for the error analysis. \\
%
%\medskip
%
%Applied to the linear problem
%\begin{equation}\label{eq:prob-lin}
%u'(t) = Au(t) + g(t), \quad u(t_0)=u_0,\qquad t_0\le t\le T
%\end{equation}
%the exponential Taylor method \eqref{eq:series_p} assumes the form
%\begin{equation}\label{eq:method-gen}
%u_{n+1} = \ee^{h A}u_n + \sum_{k=1}^p h^k \phi_k(h A) g^{(k-1)}(t_n),\quad n\ge 0.
%\end{equation}
%Its error $e_{n+1} = u_{n+1}-u(t_{n+1})$ satisfies the recursion
%$$
%e_{n+1} = \ee^{hA} e_n - \delta_{n+1},
%$$
%where
%$$
%\delta_{n+1} = \int_0^h \ee^{(h-\tau)A}
%\int_0^\tau \frac{(\tau-\xi)^{p-1}}{(p-1)!}g^{(p)}(t_n+\xi)\,\dd\xi\,\dd\tau.
%$$
%This implies the following theorem whose proof is straightforward.
%\begin{thm}
%Let the inhomogeneity $g$ in \eqref{eq:prob-lin} be $p$ times differentiable with
%$g^{(p)}\in L^1(0,T)$.
%Then, the exponential Taylor method \eqref{eq:method-gen} is convergent of order $p$.
%\end{thm}
%We remark that this convergence result extends to variable step sizes in an obvious way.
%
%
%
\subsubsection{Bounds on $\|W_N\bar\veps_N\|$: Bessel basis functions}
If the basis functions are Bessel functions
or modified Bessel functions of the first kind, further analysis is required
to show that $\|W_N\bar\veps_N\|$ vanishes.
%Let $\veps_{N,k}(t)$, $k=0,\ldots,N-1$ denote
%the elements of $\bar\veps_N(t)\in\RR^N$.
The elements of $\bar\veps_N$, denoted by $\veps_{N,k}$, $k=0,\ldots,N-1$, can be bounded as follows.
%From the definition of the expansion we have that
%\begin{equation}\label{eq:gWexpe1}
% g(t)-W\ee^{sH_N}e_1=\sum_{\ell=0}^{N-1}w_\ell\veps_{N,\ell}(t)+
%\sum_{\ell=N}^\infty w_\ell J_\ell(t),
%\end{equation}
%where
%\begin{equation} \label{def:veps_N}
%\begin{pmatrix}
%\veps_{N,0}(t)\\
%\vdots\\
%\veps_{N,N-1}(t)\\
%\end{pmatrix}=
%\bar\veps_{N}(t):=
%  \bar{J}_N(t) 
%-\exp(tH_N)e_1,\;\;
%\bar{J}_N(t)=\begin{pmatrix}
%  J_0(t)\\
%  \vdots\\
%  J_{N-1}(t)
% \end{pmatrix}.
%\end{equation}
%The characterization
%of convergence is based on
%separating the two terms in the right-hand side of \eqref{eq:gWexpe1}.
%The second term is the remainder term of the Bessel expansion. 
%The first term is bounded by the following results.
%
%

\begin{lem}\label{lem:lemma_bessel_exp_2}
Let $H_N \in \RR^{N \times N}$ be defined as either \eqref{eq:H_bessel} or \eqref{eq:H_bessel_modified}.
Then, the vector $\bar\veps_N$ satisfies
\begin{equation}
\bar\veps_N(t)=
\int_0^t\ee^{(t-s)H_N} J_N(s)e_N \, \dd s \label{eq:lemma_bessel_exp_2_exact}
\end{equation}
and is bounded as follows.
\begin{itemize}
\item[(a)] For all $t\geq 0$,
%There exists a constant $C>0$ such that 
\begin{equation} \label{eq:lemma_bessel_exp_2_bound}
\norm{\bar\veps_N(t)} \leq  \frac{ \left( \frac{1}{2} t\right)^{N} }{ \, (N+1)!} \sqrt{2} t \ee^t.
\end{equation}
%for all $t\ge 0$.
\item[(b)] Suppose $t\ge 2$. Then there
exists a constant $\widetilde C(t)$ depending only on $t$ such that
 for $1\leq k \leq N$, 
\begin{equation} \label{eq:veps_N_k_bound}
\abs{\veps_{N,k}(t)} \leq \widetilde C(t)  \frac{t^{k}}{2^{2N-k}} \frac{(N-k)!}{(2N-k+1)!}.
\end{equation}
%where $\veps_{N,k}(t)$ is the $k$th element of $\bar\veps_N(t)\in\RR^N$.
\end{itemize}
%where $C$ is a constant.
\end{lem}

\begin{proof}
Proof of \eqref{eq:lemma_bessel_exp_2_exact}: From the properties \eqref{eq:Bessel_properties} we
see that $\bar J_N(t)$ satisfies the initial value problem
$\bar J_N'(t) = H_N \bar J_N(t) + J_N(t) e_N$, $J(0) = e_1$.
Therefore, the error $\bar\veps_N(t) := \bar J_N(t) - \ee^{ t H_N} e_1$
satisfies the initial value problem
$\bar\veps_N'(t) = H_N \bar\veps_N + J_N(t) e_N$, $\bar\veps_N(0) = 0$,
%Using the integral equation \eqref{eq:integral_equation} we see that
for which the solution is given by \eqref{eq:lemma_bessel_exp_2_exact}.
%$$
%\bar\veps_N(t) = \int\limits_0^t \ee^{(t-s) H_N} J_N(s) e_N \, \dd s.
%$$

\noindent The statement \eqref{eq:lemma_bessel_exp_2_bound}
follows from properties of $H_N$ and Bessel functions as follows.
From Lemma~\ref{lem:lemma_bessel_exp_1}, we have
that $\norm{\ee^{(t-s) H_N}} \leq \sqrt{2} \, \ee^{t-s}$
and therefore 
\begin{equation*}
%\begin{aligned}
\norm{\bar\veps_N(t)} \leq \int\limits_0^t \norm{\ee^{(t-s) H_N}} \abs{J_N(s)} \, \dd s 
 \leq \sqrt{2} \int\limits_0^t \ee^{(t-s)} \frac{(\frac{1}{2}s)^N}{N!} \ee^s  \, \dd s 
 =  \frac{ \left( \frac{1}{2} t\right)^{N} }{ \, (N+1)!} \sqrt{2} \, t \ee^t,
%\end{aligned}
\end{equation*}
where in the last step we used that for $t>0$,  $|\phi_N(t)|\le \frac{\abs{\frac{1}{2}t}^N}{N!}\ee^{t}$
if $\phi_N=J_N$ or $\phi_N=I_N$, which is a consequence  of
the formula \cite[pp.\;49]{Watson}, 
%We know that for $J_n(z)$ of integer order $n$ )
\begin{equation} \label{eq:Bessel_bound}
 \abs{J_n(z)} \leq \frac{\abs{\frac{1}{2}z}^n}{n!}\ee^{\abs{\mathrm{Im}(z)}}.
\end{equation}
%Therefore, for $t>0$,
%$$
%\abs{J_n(t)} \leq \frac{(\frac{1}{2}t)^n}{n!} \quad \textrm{and} \quad \abs{I_n(t)} \leq \frac{(\frac{1}{2}t)^n}{n!} \ee^t.
%$$

%by using \eqref{eq:Bessel_bound},
%

It remains to show  \eqref{eq:veps_N_k_bound}.
We first note that, 
\eqref{eq:Bessel_bound} and 
Lemma~\ref{lem:element_bound} with $R=t^2$ implies that 
\begin{equation} \label{eq:beta_fct}
\abs{\veps_{N,k}(t)} = \abs{  \int\limits_0^t e_k^{\mathrm{T}} \ee^{(t-s)H_N} e_N J_N(s) \, \dd s } \leq
\frac{ \ee^t C(t^2)}{ t^{2(N-k)} N! \, 2^{2N-k} } \int\limits_0^t s^N (t-s)^{N-k} \, \dd s,
\end{equation}
where  $C(t^2)$ is given by \eqref{eq:CR}.
%%By applying Lemma~\ref{lem:element_bound} to the matrix $\ee^{(t-s)H_N}$,
%% and choosing $R=t^2$, we get
%\begin{equation}\label{eq:eken_bound}
% \abs{ e_k^{\mathrm{T}} \ee^{(t-s)H_N} e_N   }
%\leq C(t^2) \left( \frac{t-s}{ 2 t^2} \right)^{N-k}.
%\end{equation}
%where $C(R)$ is defined by \eqref{eq:CR}.
%We know that for $J_n(z)$ of integer order $n$ (see~\cite[pp.\;49]{Watson})
%\begin{equation} \label{eq:Bessel_bound}
% \abs{J_n(z)} \leq \frac{\abs{\frac{1}{2}z}^n}{n!}\ee^{\abs{\mathrm{Im(z)}}}.
%\end{equation}
%By combining \eqref{eq:eken_bound} and \eqref{eq:Bessel_bound} we find that 
%$$
%\abs{e_k^{\mathrm{T}} \ee^{(t-s)H_N} e_N J_N(s) } \leq \frac{ \wt C(t)}{ t^{2(N-k)}}  \frac{s^N (t-s)^{N-k}}{N! \, 2^{2N-k} },
%$$
%where $\widetilde C(t) = \ee^t C(t^2)$.
We identify the integral on right-hand side of \eqref{eq:beta_fct} as a scaled 
Beta function $t^{m+n+1}B(m+1,n+1)$.
The conclusion \eqref{eq:veps_N_k_bound} now follows from 
the application of a formula for $B(m+1,n+1)$ in \cite[pp.\;258]{Stegun:1964:HANDBOOK}.
More precisely,
\begin{align*}
\abs{\veps_{N,k}(t)} 
&\le  \frac{ \ee^t C(t^2)}{ t^{2(N-k)} N! \, 2^{2N-k} } \frac{t^{2N-k+1} N! (N-k)!}{(2N-k+1)!}
=  \ee^t C(t^2)  \frac{t^{k+1}}{2^{2N-k}} \frac{(N-k)!}{(2N-k+1)!}.
%\int\limits_0^t  C(t) \frac{s^N (t-s)^{N-k}}{2^{2N-k} (N+1)!} \, \dd s \\
%& \leq C(t) \frac{N!(N-k)!}{(2N-k)!2^{2N-k} (N+1)!} \\ &= C(t) \frac{(N-k)!}{(2N-k)!2^{2N-k} (N+1)}.
\end{align*}
\end{proof}

We have now derived a bound on  $\bar\veps_N$ and shown that $\|\bar\veps_N(t)\|\rightarrow 0$
as $N \rightarrow \infty$, when the basis functions are the Bessel functions
or the modified Bessel functions of the first kind.
Note that this does not necessarily imply that 
$\|W_N\bar\veps_N(t)\|\rightarrow 0$, since the coefficient
matrix $W_N$ may not be bounded for all $N$. %can be unbounded. %In fact, the coefficient matrix $W_N$ is often unbounded. 
Fortunately, the analyticity of $g(t)$ 
gives us a bound on the growth of the coefficients $w_k$.
%gives us a bound on how fast the coefficients can grow. 

%In order to bound $W_N\bar\veps_N$ %$\sum_{\ell}w_\ell\veps_{N,\ell}$
%we also need to bound $w_\ell$, i.e.,  the columns of $W_N$. Under
%the assumption of analyticity we have the following bound.
%
\begin{lem} \label{lem:w_k_bound}
Suppose $g$ is analytic in a neighborhood of a disc of radius $t$ centered at the origin.
%inside and on a circle of radius $t$. 
Let $M_t$ be defined as
\begin{equation} \label{eq:M_t}
M_t = \max_{\abs{\lambda} = t} \norm{g(\lambda)}.
\end{equation}
Let the vectors $w_k$ be the coefficients of the expansion \eqref{eq:expansion} of $g(t)$, where the
functions $\phi_\ell$ are the Bessel or the modified Bessel functions of the first kind.
Then, for $0 \leq t < 2$ we have the bound
\begin{equation} \label{eq:w_k_bound_0}
\norm{w_k} \leq M_t \, k!  \left( \frac{2}{t} \right)^k \quad \textrm{for all} \quad k \geq 0.
\end{equation}
Moreover, for $t\geq 2$, we have the bounds
\begin{equation} \label{eq:w_k_bound_1}
\norm{w_k} \leq M_t k! \quad \textrm{for all} \quad k \geq 0,
\end{equation}
and
\begin{equation} \label{eq:w_k_bound_2}
\norm{w_k} \leq M_t \, k! 2  \left( \frac{2}{t} \right)^k \quad \textrm{for all} \quad k > \left(\frac{t}{2} \right)^2 +1.
\end{equation}
\end{lem}
\begin{proof}
% Then, the expansion coefficients
% are bounded by 
% \begin{equation} \label{eq:w_k_bound_3}
% \norm{w_k} \leq \frac{M_t}{2} \sum\limits_{\ell=0}^{\left \lfloor 
% \frac{k}{2} \right \rfloor} \frac{k(k-\ell-1)!}{\ell !} \left( \frac{2}{t} \right)^{k-2\ell}.
% \end{equation}
%From \eqref{eq:Neumann}  coefficients $w_k$ of the Bessel expansion \eqref{eq:Bessel_expansion}
%are given by expressions of the form (eqref here)
The closed form  \eqref{eq:Neumann}  for the coefficients $w_k$ implies that for all $k \geq 0$
\begin{equation} \label{eq:Neumann_polys}
\|w_k\| \le  \sum\limits_{\ell=0}^{\left \lfloor \frac{k}{2} \right \rfloor} \frac{k(k-\ell-1)!}{\ell !} 2^{k-2\ell - 1}  
\| \frac{1}{2 \pi \ii} \int\limits_{\Gamma} \frac{g(\lambda)}{\lambda^{k-2\ell+1}} \, \dd \lambda \|.
%\le
%\frac{M_t}{2} \sum\limits_{\ell=0}^{\left \lfloor 
%\frac{k}{2} \right \rfloor} \frac{k(k-\ell-1)!}{\ell !} \left( \frac{2}{t} \right)^{k-2\ell}
\end{equation}
%which converges to zero as $N \rightarrow \infty$.
%By \eqref{eq:M_t} we get the bound
Combining \eqref{eq:M_t} and \eqref{eq:Neumann_polys} gives us
%$\norm{ \frac{1}{2 \pi \ii} \int_{\Gamma} \frac{g(\lambda)}{\lambda^{k-2\ell+1}} \, \dd \lambda }
%\leq M_t/t^{k-2\ell}$, which in combination with \eqref{eq:Neumann_polys} gives
\begin{equation} \label{eq:w_k_bound_3}
\|w_k\| \le  \frac{M_t}{2} \sum\limits_{\ell=0}^{\left \lfloor \frac{k}{2} \right \rfloor} 
\frac{k(k-\ell-1)!}{\ell !} \left( \frac{2}{t} \right)^{k - 2\ell}.
\end{equation}
Thus, when $t<2$, we have that for all $k\geq 0$
$$
\norm{w_k} \leq \frac{M_t}{2} \left( \frac{2}{t} \right)^k \sum\limits_{\ell=0}^{\left \lfloor \frac{k}{2} \right \rfloor} 
\frac{k(k-\ell-1)!}{\ell !} 
%\leq \frac{M_t}{2} \left( \frac{2}{t} \right)^N \left(k! + 
%\sum\limits_{\ell=0}^{\left \lfloor \frac{k}{2} \right \rfloor-1} 
%\frac{k(k-\ell-1)!}{\ell !} \right) 
\leq \frac{M_t}{2} \left( \frac{2}{t} \right)^k 2 k!
$$
which gives \eqref{eq:w_k_bound_0}. When $t\geq 2$, $\left( \frac{2}{t} \right)^{k-2\ell} \leq 1$
if $0\leq \ell \leq \left \lfloor \frac{k}{2} \right \rfloor$, 
and with a similar reasoning \eqref{eq:w_k_bound_1} follows from \eqref{eq:w_k_bound_3}.
%proves \eqref{eq:w_k_bound_0}.
%where the countour $\Gamma$ goes counterclockwise along a
%circle of radius $t$ centered at the origin.
%Thus, using \eqref{eq:Neumann_polys} we get
%\begin{equation} \label{eq:w_k_bound_3}
%\norm{w_k} \leq \frac{M_t}{2} \sum\limits_{\ell=0}^{\left \lfloor 
%\frac{k}{2} \right \rfloor} \frac{k(k-\ell-1)!}{\ell !} \left( \frac{2}{t} \right)^{k-2\ell},
%\end{equation}
%where we used that the assumption \eqref{eq:M_t} gives us the bound
%
%When  $t \geq 2$, we also have that $\frac{k(k-\ell-1)!}{\ell !} \left( \frac{2}{t} \right)^{k-2\ell} \leq k!$ if 
%$0\leq \ell \leq \left \lfloor \frac{k}{2} \right \rfloor$ and the bound \eqref{eq:w_k_bound_1} follows from \eqref{eq:w_k_bound_0}.

In order to show \eqref{eq:w_k_bound_2}, we note
that \eqref{eq:w_k_bound_3} can be expressed as 
\begin{equation} \label{eq:w_k_bound_4}
 \norm{w_k} \leq \frac{M_t}{2} \sum\limits_{\ell=0}^{\left \lfloor \frac{k}{2} \right \rfloor} c_\ell,
\end{equation}
where 
$c_0=k! \left( \frac{2}{t} \right)^k$ and 
$c_\ell=\frac{1}{(k-\ell) \ell} \left(\frac{t}{2}\right)^2 \cdot c_{\ell-1}$
if $\ell\geq 1$.
%\begin{equation} \label{eq:c_ell_cases}
% \begin{cases}
%    c_\ell &= \frac{1}{(k-\ell) \ell} \left(\frac{t}{2}\right)^2 \cdot c_{\ell-1}, \quad \textrm{if} \quad \ell \geq 1, \\
%    c_0 &= k! \left( \frac{2}{t} \right)^k.
% \end{cases}
%\end{equation}
When $1 \leq \ell \leq \left \lfloor \frac{k}{2} \right \rfloor$, 
$(k- \ell) \ell \geq k -1$, and we see that 
$c_\ell$ satisfies $c_\ell\le a_k c_{\ell-1}$ 
such that $c_\ell\le a_k^\ell c_0$,
where $a_k = t^2/4(k-1)$. 
%and we see from \eqref{eq:c_ell_cases} that
%$$
%\norm{w_k} \leq \frac{M_t}{2} \sum\limits_{\ell=0}^{\left \lfloor \frac{k}{2} \right \rfloor} a_k^\ell c_0,
%$$
%where $a_k = \frac{1}{k-1} \left( \frac{t}{2} \right)^2$. 
The conclusion \eqref{eq:w_k_bound_2}
follows from \eqref{eq:w_k_bound_4} and the fact that the assumption 
$k> 2\left( \frac{t}{2} \right)^2 + 1$ implies that $a_k<1/2$ and by taking the limit $\ell\rightarrow\infty$.
%Moreover, when $a_k < 1/2$, i.e., when $k> 2\left( \frac{t}{2} \right)^2 + 1$, we see that
%$$
%\norm{w_k} \leq \frac{M_t}{2 }\sum\limits_{\ell = 0}^\infty \left( \frac{1}{2} \right)^\ell c_0 
%= \frac{M_t}{2 } \frac{c_0}{1-\frac{1}{2}} = M_t \, c_0 = M_t \, k! \left( \frac{2}{t} \right)^k.
%$$
\end{proof}

By combining the bound of $\bar\veps_N$ and the 
bound of $w_\ell$ we arrive at the following result
for $W_N\bar\veps_N(t)$.

\begin{thm} \label{thm:convergence_bessel}
If $W_N$ corresponds to the expansion of $g(t)$ with the Bessel functions or the 
modified Bessel functions of the first kind, then for all $t>0$
\begin{equation}% \label{eq:error_convergence}
%\sum\limits_{\ell=0}^N w_\ell \veps_{N,\ell}(t) 
\norm{ W_N\bar\veps_N(t) } \rightarrow 0, \quad \textrm{as} \quad N \rightarrow \infty.
\end{equation}
Consequently, if the basis functions are the Bessel functions or the modified Bessel functions,
the truncation error $\|u(t)-u_N(t)\|\rightarrow 0$ independent of $t$.
\end{thm}
\begin{proof}
Consider first the case $0 \leq t < 2$. By \eqref{eq:lemma_bessel_exp_2_bound} and
\eqref{eq:w_k_bound_0} we see that
\begin{equation*}
\begin{aligned}
& \norm{ \sum\limits_{\ell=0}^N w_\ell \veps_{N,\ell}(t) }\leq \norm{\bar\veps_N(t)} \sum\limits_{k=0}^N \norm{w_k}
\leq \norm{\bar\veps_N(t)} \sum\limits_{k=0}^N M_t k! \left( \frac{2}{t} \right)^k \\
&\leq \frac{ \left(\frac{t}{2}\right)^N}{(N+1)!} \sqrt{2} t \ee^t M_t \left( \frac{2}{t} \right)^N 
\big( N! + (N-1)! + \ldots + 1 \big) \\
& =   \sqrt{2} t M_t \left(\frac{1}{N+1} + \frac{1}{(N+1)N} + 
\ldots + \frac{1}{(N+1)!} \right) \leq   \sqrt{2} t M_t \left(2 \frac{1}{N+1} \right).
\end{aligned}
\end{equation*}
Consider the case $t\geq 2$. Suppose $N > \widetilde k := \left \lceil 2 \left( \frac{t}{2} \right)^2 + 1 \right \rceil$.
We see that
\begin{equation} \label{eq:split_t_geq_2}
\norm{ \sum\limits_{\ell=0}^N w_\ell \veps_{N,\ell}(t) }
%\left\lVert  \sum\limits_{\ell=0}^N w_\ell \veps_{N,\ell}(t) \right\rVert
\leq \sum\limits_{\ell=0}^{\wt k} \norm{w_\ell} \abs{\veps_{N,\ell}(t)}  
+  \sum\limits_{\ell= \wt k + 1}^N \norm{w_\ell} \abs{\veps_{N,\ell}(t)}.
\end{equation}
We now show that both of the terms in the right-hand side of 
\eqref{eq:split_t_geq_2} vanish as $N\rightarrow\infty$.
Using the bound \eqref{eq:w_k_bound_1} of Lemma~\ref{lem:w_k_bound}, and the
bound \eqref{eq:veps_N_k_bound} of Lemma~\ref{lem:lemma_bessel_exp_2}
with $k=\ell$, we see that there exists a constant
$\wt C_2(t):=M_t\wt C(t)t^{\wt k}$, which are independent of $N$, such that
\begin{equation}\label{eq:split_t_geq_2b}
\sum\limits_{\ell=0}^{\wt k} \norm{w_\ell} \abs{\veps_{N,\ell}(t)} 
\leq  \wt C_2(t) \sum\limits_{\ell=0}^{\wt k} \frac{\ell! (N-\ell)!}{2^{2N - \ell} (2N - \ell +1)!}.
%\leq  \wt C_2(t) \sum\limits_{\ell=0}^{\wt k} \frac{(\ell+1)! (N-\ell)!}{2^{2N - \ell} (2N - \ell +1)!}.
\end{equation}
Since for $\ell \leq \wt k \leq N$, $\ell! (N-\ell)!<(\ell+1)! (N-\ell)! \leq (N+1)!$ and 
$2^{2N - \ell} (2N - \ell +1)! \geq 2^N ((N- \wt k) + N + 1)!$,
we see that 
\begin{equation}\label{eq:split_t_geq_2b_new}
\frac{(\ell+1)! (N-\ell)!}{2^{2N - \ell} (2N - \ell +1)!} \leq
\frac{(N+1)!}{2^N((N-\wt k)+N+1)!}
\leq \frac{1}{2^N N^{N- \wt k}}.
\end{equation}
By inserting \eqref{eq:split_t_geq_2b_new} into \eqref{eq:split_t_geq_2b} we 
conclude that the first term in \eqref{eq:split_t_geq_2} vanishes as $N\rightarrow\infty$.
%\begin{equation}\label{eq:split_t_geq_2c}
%\sum\limits_{\ell=0}^{\wt k} \norm{w_\ell} \abs{\veps_{N,\ell}(t)}  \rightarrow 0, \quad \textrm{as} \quad N \rightarrow \infty.
%\end{equation}

For the second term of \eqref{eq:split_t_geq_2}, we use the bound \eqref{eq:w_k_bound_2} of Lemma~\ref{lem:w_k_bound}, and the
bound \eqref{eq:veps_N_k_bound}, to see 
that there exists a constant $\wt C_3(t):=M_t\wt C(t)$, which are independent of $N$, such that
\begin{equation*}
 \begin{aligned}
  \sum\limits_{\ell=\wt k +1}^N \norm{w_\ell} \abs{\veps_{N,\ell}(t)}  &
  \leq C_3(t) \sum\limits_{\ell= \wt k +1}^N   \left( \frac{2}{t} \right)^\ell \frac{t^\ell}{2^{2N - \ell}} 
  \frac{\ell ! (N-\ell)!}{(2N - \ell +1)!} \\
& \leq C_3(t) \sum\limits_{\ell= \wt k +1}^N \frac{ N !}{(2N - \ell +1)!} \\
%& = C_3(t)\left( \sum\limits_{\ell= \wt k +1}^{N-1} \frac{N!}{(2N - \ell +1)!} + \frac{1}{N+1}\right) \\
& \leq C_3(t) \left((N-\wt k-1) \cdot \frac{1}{(N+2)(N+1)} + \frac{1}{N+1}\right),
 \end{aligned}
\end{equation*}
This implies that the second term in the right-hand side of \eqref{eq:split_t_geq_2} 
converges to zero as $N \rightarrow \infty$ and completes the proof.
%Since we showed that the first term converges to zero in \eqref{eq:split_t_geq_2c},  \eqref{eq:split_t_geq_2} vanishes as $N\rightarrow\infty$,
%which completes the proof.
\end{proof}

\subsection{Error bounds for the Arnoldi approximation}

In order to show convergence of \eqref{eq:error_splitting} we 
will now study the second term in \eqref{eq:error_splitting}.
%Next, we give results for the convergence of the second term in \eqref{eq:error_splitting}, 
%i.e., bounds for the error arising from the Arnoldi approximation of the matrix exponential.
%
%% This is all repetition 
%Let $A_N \in \CC^{(n+N) \times (n+N)}$ and $u_N$ be defined as above, i.e.,
%\begin{equation} \label{def:A_N}
% A_N=
% \begin{bmatrix}
%   A &W_N\\
%    0 & H_N
% \end{bmatrix}\in\CC^{(n+N)\times(n+N)}, \quad u_N = \begin{bmatrix} u_0 \\ e_1 \end{bmatrix},
%\end{equation}
%where $u_0 \in \CC^n$, $A  \in \CC^{n\times n}$, 
%$W_N \in \CC^{n \times N}$ and  $H_N \in \CC^{N \times N}$ is a Hessenberg matrix of the form
%\eqref{eq:H_jordan}, \eqref{eq:H_bessel} or \eqref{eq:H_bessel_modified}.
%We assume that $A_N$ arises from a linearization of the ODE \eqref{eq:semilinearODE}
%using a Taylor or Bessel expansion (using Bessel or modified Bessel functions of first order). 
%Thus $W_N$ is asumed to be of the form 
%$$
%W_N = G_N K_N(H_N,e_1)^{-1}, 
%$$
%where $G_N = \begin{bmatrix} g(0) & g'(0) & \ldots & g^{(N-1)}(0) \end{bmatrix}$ and the vectors
%$g(0), g'(0), \ldots, g^{(N-1)}(0)$ correspond to the derivatives of $g$ in \eqref{eq:semilinearODE}
%and $K_N(H_N,e_1)$ is the Krylov matrix as defined in 
%\eqref{def:krylov_matrix} (see the discussion in Section \ref{sect:otherexps}).
%
Let
$$
Q_N = \begin{bmatrix} Q_{1,N+1} \\ Q_{2,N+1} \end{bmatrix} \in \CC^{(n+N+1) \times (N+1)}, 
$$
where $Q_{1,N} \in \CC^{n \times N}$ is the orthonormal matrix
and $F_N = Q_N^* A_N Q_N$ the Hessenberg matrix given by the infinite Arnoldi algorithm after $N$ iterations.
The Arnoldi relation \eqref{eq:arnoldirelation}, with $B=A_N$, implies that
\begin{equation} \label{eq:Arnoldi_relation}
\begin{aligned}
   A Q_{1,N} + W Q_{2,N} &= Q_{1,N} F_N + f_{N+1,N} q_{1,N+1} e_N^{\mathrm T} \\
   H_N Q_{2,N} &= Q_{2,N} F_N + f_{N+1,N} q_{2,N+1} e_N^{\mathrm T}.
 \end{aligned}
\end{equation}
The polynomial approximation property of the Arnoldi 
method \cite[Lemma~3.1]{Saad}
states that for any polynomial $p$ of degree less than $N$
we have 
$p(A_N)u_N= \beta Q_Np(F_m)e_1$. 
In our situation we can exploit the
structure of $A_N$ when we select $p(z)=z^\ell$.
From the second block of $p(A_N)u_N$ we conclude that 
%From the polynomial approximation property of the Arnoldi approximation  
%it follows that
$H_N^\ell e_1 \beta^{-1} = Q_{2,N} F_N^\ell e_1 \quad \textrm{for all} \quad \ell \leq N-1$.
By stacking this equation as columns
into a matrix equation we find that 
$K_N(H_N,e_1)\beta^{-1} = Q_{2,N} K_N(F_N,e_1)$,
such that 
\begin{equation} \label{eq:Q2_relation}
Q_{2,N} = K_N(H_N,e_1) K_N(F_N,e_1)^{-1} \beta^{-1},
\end{equation}
where $K_N$ denotes the Krylov matrix, defined in \eqref{def:krylov_matrix}.
The orthonormality of $Q_N$ implies that $\norm{Q_N} = 1$,
from which it follows that
$\norm{Q_{1,N}} \leq 1$  and $\norm{Q_{2,N}} \leq 1$.
% Let us consider first the monomial case, i.e., when $W_N = G_N$ (see (?)), and $H_N$ has subdiagonals 1 and is otherwise zero.
% Inspecting the powers of $A_N$ (see \eqref{def:A_N}), one easily verifies that
% $$
% \begin{bmatrix}
%   I_n & 0
% \end{bmatrix} \exp( t A_N) u_N = \exp(t A) u_0 + \sum\limits_{\ell=1}^N t^\ell \varphi_\ell( t A ) g^{(\ell-1)}(0),
% $$
% where the $\varphi$ functions are given as
% \begin{equation} \label{def:phi_function}
% \varphi_\ell(z) = \sum\limits_{k=0}^\infty \frac{z^k}{(k + \ell)!} = \int\limits_0^1 \ee^{(1-\tau) z} \frac{\tau^{\ell-1}}{(\ell-1)!}.
% \end{equation}
Consider $A_N$ of the form \eqref{eq:Adef} for a general Hessenberg matrix
$H_N$.
%i.e. $W_N = G_N K_N(H_N,e_1)^{-1}$). 
The infinite Arnoldi approximation at step $N$ is given by
\begin{equation} \label{eq:inf_approx}
%\begin{bmatrix}
%  I_n & 0
%\end{bmatrix} \exp( t A_N) u_N \approx
\begin{bmatrix}
  I_n & 0
\end{bmatrix}
 Q_N \exp( t F_N) e_1 \beta = Q_{1,N} \exp( t F_N) e_1 \beta,
\end{equation}
where $\beta = \norm{u_N}$. 
We again use the polynomial approximation 
property, which implies
that 
$\sum_{\ell=0}^{N-1}\frac{1}{\ell!}A_N^\ell=
Q_N\sum_{\ell=0}^{N-1}\frac{1}{\ell!}F_N^\ell\beta$. 
%Due to the polynomial approximation property of 
%the Arnoldi approximation \eqref{eq:inf_approx}, the error of the approximation can be expressed as
Hence, the second term in the error \eqref{eq:error_splitting} can be
expressed as
\begin{equation}\label{eq:error_expression}
u_N(t)-u_N^{IA}(t)=
\begin{bmatrix}
  I_n & 0
\end{bmatrix} \left( \exp( t A_N) u_N -  Q_N \exp( t F_N) e_1 \beta \right) 
=a_N+b_N,
\end{equation}
where 
\begin{eqnarray}
 a_N&:=&
%\begin{bmatrix}
%  I_n & 0
%\end{bmatrix} \exp( t A_N) u_N=
\begin{bmatrix}
  I_n & 0
\end{bmatrix} r_N(t A_N) u_N \label{eq:aNdef}\\
b_N&:=&
%-\begin{bmatrix}
%  I_n & 0
%\end{bmatrix} Q_N \exp( t F_N) e_1 \beta=
- Q_{1,N} r_N(t F_N) e_1 \beta\label{eq:bNdef}
\end{eqnarray}
%\begin{align} 
%u_N(t)-u_N^{IA}(t) =
%%\mathrm{err}_N  := 
%& \begin{bmatrix}
%  I_n & 0
%\end{bmatrix} \left( \exp( t A_N) u_N -  Q_N \exp( t F_N) e_1 \beta \right) \nonumber \\
%% = &\begin{bmatrix}
%%   I_n & 0
%% \end{bmatrix} \left( r_N(t A_N) u_N - Q_N r_N(t F_N) e_1 \beta, \right) \\
%= &\begin{bmatrix}
%  I_n & 0
%\end{bmatrix} r_N(t A_N) u_N - Q_{1,N} r_N(t F_N) e_1 \beta, \label{eq:error_expression}
%\end{align}
%where $r_N$ denotes the remainder:
and $r_N$ denotes the remainder term in 
the truncated Taylor expansion. We will 
use an explicit representation of $r_N$, 
\begin{equation}\label{eq:rN}
r_N(z) = \sum\limits_{\ell=N}^\infty \frac{z^\ell}{\ell !} = z^N \varphi_N(z),
\end{equation}
with the standard definition of $\varphi$-functions,
\begin{equation} \label{def:phi_and_remainder}
  \varphi_{\ell}(z) := \sum\limits_{k=0}^\infty \frac{z^k}{(k + \ell)!} 
  = \int\limits_0^1 \ee^{(1-\tau) z} \frac{\tau^{\ell-1}}{(\ell-1)!}.
\end{equation}
%We will next inspect separately the convergence of the two terms on the right-hand side of 
%\eqref{eq:error_expression}, starting with $\begin{bmatrix}
%  I_n & 0
%\end{bmatrix} r_N(t A_N) u_N$.

\subsubsection{Convergence of $a_N$ in \eqref{eq:error_expression}}
% Inspecting the terms $A_N^\ell u_N$, we get the following lemma.
The analysis of \eqref{eq:error_expression} is separated into analysis of $a_N$ and $b_N$. We first need 
a reformulation of $a_N$.
 \begin{lem} \label{lem:remainder_A_N}
Let $A_N$, $u_N$, $r_N$ and  $\varphi_\ell$ be defined as above. Then, the following expression holds
 \begin{equation} \label{eq:remainder_expression}
  \begin{aligned}
%  \begin{bmatrix}
%    I_n & 0
%  \end{bmatrix} r_N(tA_N) \begin{bmatrix} u_0 \\ e_1 \end{bmatrix} 
   a_N &= (tA)^N \varphi_N(tA) u_0 + \sum\limits_{\ell=1}^N t^\ell (tA)^{N-\ell} \varphi_N(t A) g^{(\ell-1)}(0) \\
 & \,\,  + \sum\limits_{\ell=N+1}^\infty t^\ell \varphi_\ell (t A) W_N H_N^{\ell-1} e_1.
\end{aligned}
 \end{equation}
% where the $\varphi$ functions are defined as
 \end{lem}
 \begin{proof}
 By induction it is readily verified from \eqref{eq:Adef} that
 $$
 A_N^k \begin{bmatrix} u_0 \\ e_1 \end{bmatrix} 
 = \begin{bmatrix} A^k u_0 + \sum\limits_{\ell=1}^k A^{k-\ell} W_N H_N^{\ell - 1}e_1 \\ H_N^k e_1 \end{bmatrix}.
 $$
 From this it follows that %and definition \eqref{def:phi_and_remainder} it follows that
 \begin{equation} \label{eq:split_first_term}
 \begin{aligned}
 &  \begin{bmatrix}
   I_n & 0
 \end{bmatrix} r_N(tA_N) \begin{bmatrix} u_0 \\ e_1 \end{bmatrix}  = 
 \begin{bmatrix} I_n & 0 \end{bmatrix} \sum\limits_{k=N}^\infty \frac{(tA_N)^k}{k!} 
 \begin{bmatrix} u_0 \\ e_1 \end{bmatrix}  \\
%& = \sum\limits_{k=N}^\infty \frac{t^k}{k!} \left( A^k u_0 + \sum\limits_{\ell=1}^k  A^{k-\ell} W_N H_N^{\ell-1}e_1 \right) \\
 &= \sum\limits_{k=N}^\infty \frac{(tA)^k}{k!}u_0 
 + \sum\limits_{k=N}^\infty \sum\limits_{\ell=1}^k t^\ell \frac{(tA)^{k-\ell}}{k!} W_N H_N^{\ell-1}e_1  \\
&= \sum\limits_{k=N}^\infty \frac{(tA)^k}{k!}u_0 
 + \sum\limits_{\ell=1}^N \sum\limits_{k=N}^\infty t^\ell \frac{(tA)^{k-\ell}}{k!} W_N H_N^{\ell-1}e_1 
 + \sum\limits_{\ell=N+1}^\infty \sum\limits_{k=\ell}^\infty t^\ell \frac{(tA)^{k-\ell}}{k!} W_N H_N^{\ell-1}e_1.
\end{aligned}
\end{equation}
Since $W_N H_N^{\ell - 1} e_1 = g^{(\ell - 1)}(0)$ when $0 \leq \ell \leq N$, we find for the second term
on the last line of \eqref{eq:split_first_term} that
\begin{equation*}
%\begin{aligned}
\sum\limits_{\ell=1}^N \sum\limits_{k=N}^\infty t^\ell \frac{(tA)^{k-\ell}}{k!} W_N H_N^{\ell-1}e_1  =
%\sum\limits_{k=0}^\infty \frac{(tA)^k}{(k+N)!} g^{(\ell-1)}(0) \\ & =
  \sum\limits_{\ell=1}^N t^\ell (tA)^{N-\ell} \varphi_N(t A) g^{(\ell-1)}(0).
%\end{aligned}
\end{equation*}
For the third term on the last line of \eqref{eq:split_first_term} we see that
\begin{equation*}
%\begin{aligned}
\sum\limits_{\ell=N+1}^\infty \sum\limits_{k=\ell}^\infty t^\ell \frac{(tA)^{k-\ell}}{k!} W_N H_N^{\ell-1}e_1  
%& = \sum\limits_{\ell = N+1}^\infty t^\ell \sum\limits_{k=0}^\infty \frac{(tA)^k}{(k+\ell)!} W_N H_N^{\ell-1} e_1 \\
 = \sum\limits_{\ell=N+1}^\infty t^\ell \varphi_\ell (t A) W_N H_N^{\ell-1} e_1,
%\end{aligned}
\end{equation*}
from which the claim follows.
 \end{proof}

We are now ready to state convergence of the first term in \eqref{eq:error_expression}
under general assumptions about the nonlinearity $g$. 
%Lemmas~\ref{lem:remainder_bound} 
%and~\ref{lem:H_N_powers} of the appendix will be needed in the proof of Theorem~\ref{lem:first_term_convergence}.

\begin{thm} \label{lem:first_term_convergence}
Let $A_N$ be defined as in \eqref{eq:Adef}. Assume that for the vectors 
$g^{(\ell)}(0)$ are bounded by
\begin{equation} \label{eq:w_l_assumption}
\norm{g^{(\ell)}(0)} \leq c \, \norm{A}^{\ell}
\end{equation}
for some constant $c\in\RR$.
Then,  $a_N$ defined by \eqref{eq:aNdef} 
satisfies
%vanishes as $N\rightarrow\infty$,
%i.e., 
%the remainder term \eqref{eq:remainder_expression} converges to zero as $N \rightarrow \infty$.
\[
   a_N\rightarrow 0\;\;\textrm{ as }N\rightarrow\infty.
\]
\end{thm}
\begin{proof}
We bound the norm of the term \eqref{eq:remainder_expression} as
\begin{equation} \label{eq:three_terms}
\begin{aligned}
\norm{ \begin{bmatrix}
  I_n & 0
\end{bmatrix}  r_N(t A_N) u_N  }_2 & \leq \norm{ (tA)^N \varphi_N(tA) u_0 } + 
\norm{ \sum\limits_{\ell=1}^N t^\ell (tA)^{N-\ell} \varphi_N(t A) g^{(\ell-1)}(0) } \\
&+\norm{ \sum\limits_{\ell=N+1}^\infty t^\ell \varphi_\ell (t A) W_N H_N^{\ell-1} e_1}.
\end{aligned}
\end{equation}
% \begin{equation} \label{eq:err_bound_1}
% \norm{err_N} \leq \norm{\sum\limits_{\ell = 0}^N t^\ell r_{N-\ell}^\ell (t A) w_\ell } + \norm{r^0_N (t F_N) e_1} \beta.
% \end{equation}
% Using Lemma~\ref{lem:remainder_bound} we see that
% $$
% \norm{r_{N-\ell}^\ell(t A)} \leq \frac{\norm{t A}^{N-\ell+1} \ee^{t \mu(A)} }{(N+1)!},
% $$
By Lemma~\ref{lem:remainder_bound} we get a bound for the first term in \eqref{eq:three_terms} as
$$ 
\norm{(tA)^N \varphi_N(tA) u_0} \leq \frac{\norm{tA}^N \max(1,\ee^{\mu(tA)})}{N!} \norm{u_0}.
$$
For the second term in \eqref{eq:three_terms}, we see that
\begin{equation*}
\begin{aligned}
 \norm{\sum\limits_{\ell=1}^N t^\ell (tA)^{N-\ell} \varphi_N(tA) g^{(\ell-1)}(0)}   & \leq 
\sum\limits_{\ell=1}^N t^\ell \norm{tA}^{N-\ell} \norm{g^{(\ell-1)}(0)} \norm{\varphi_N(tA)} \\
 \leq  \wt C \sum\limits_{\ell=1}^N \norm{tA}^N \frac{\max(1,\ee^{\mu(tA)})}{N!}
 & =  \wt C \norm{tA}^N \frac{\max(1,\ee^{\mu(tA)})}{(N-1)!},
\end{aligned}
\end{equation*}
where $\wt C = C/\norm{tA}$. Thus, also the second term in \eqref{eq:three_terms} converges
to zero as $N \rightarrow \infty$.

For the third term in \eqref{eq:three_terms}, we use Lemmas~\ref{lem:remainder_bound} 
and~\ref{lem:H_N_powers} to find that 
\begin{equation*}
\begin{aligned}
\norm{\sum\limits_{\ell = N+1}^\infty t^\ell \varphi_\ell(tA) W_N H_N^{\ell-1} e_1 } &
\leq  \sum\limits_{\ell = N+1}^\infty t^\ell \norm{\varphi_\ell(tA)} \norm{G_N} \norm{K_N(H_N,e_1)^{-1} H_N^{\ell-1}e_1} \\
& \leq \sum\limits_{\ell = N+1}^\infty t^\ell \frac{\max(1,\ee^{\mu(tA)})}{\ell!}  \norm{G_N}  2 \sqrt{N} (1 + \sqrt{2})^N \\ 
& =     t^{N+1} \varphi_{N+1}(t)  \max(1,\ee^{\mu(tA)})2 \sqrt{N} (1 + \sqrt{2})^N \norm{G_N}  \\
& \leq  \frac{ t^{N+1} \ee^t \max(1,\ee^{\mu(tA)}) 2 \sqrt{N} (1 + \sqrt{2})^N \norm{G_N}  \ee^t}{(N+1)!}. 
\end{aligned}
\end{equation*}
By assumption \eqref{eq:w_l_assumption},
$$
\norm{G_N} \leq \norm{G_N}_F = \sqrt{ \sum\limits_{\ell = 0}^{N-1} \norm{g^{(\ell-1)}(0)}^2 } \leq 
C \sqrt{ \sum\limits_{\ell=0}^{N-1} \norm{A}^{2\ell}} = C \sqrt{\frac{\norm{t A}^{2N} - 1 }{ \norm{t A}^2 - 1}  }.
$$
Thus also the third term in \eqref{eq:three_terms} converges
to zero as $N \rightarrow \infty$.
\end{proof}

\subsubsection{Convergence of $b_N$ in \eqref{eq:error_expression}} \label{sec:krylov_error}
Bounding the remainder $Q_{1,N} r_N(t F_N) e_1 \beta$ in the error expression
\eqref{eq:error_expression} needs in general additional assumptions about $F_N$.
Before stating the convergence theorem, we need the following lemma.
\begin{lem} \label{lem:powers_bound}
Assume that $1 < \norm{H_N} < \norm{A}$ and that 
\eqref{eq:w_l_assumption} is satisfied
%for 
%$0 \leq \ell \leq N-1$
%\begin{equation}\label{eq:g_growth}
%\norm{g^{(\ell)} (0) } \leq c \, \norm{A}^{\ell}
%\end{equation}
for some constant $c > 0$. Then, for $0 \leq \ell \leq N$
\begin{equation*} %\label{eq:power_statement}
\norm{F_N^\ell e_1} \leq (1 + \beta^{-1} c \ell) \norm{A}^\ell.
\end{equation*}
\begin{proof}
From \eqref{eq:Adef}, \eqref{eq:Bessel_Krylov_relation} and \eqref{eq:Q2_relation} we see that the Hessenberg matrix $F_N$ is given by
\begin{equation*}
\begin{aligned}
F_N = Q_N^* A_N Q_N &= Q_{1,N}^* A Q_{1,N} + Q_{2,N}^* H_N Q_{2,N} + Q_{1,N}^* W_N Q_{2,N} \\
%    &= Q_{1,N}^* A Q_{1,N}  + Q_{2,N}^* H_N Q_{2,N} + \beta^{-1} Q_{1,N}^* W_N K_N(H_N,e_1) K_N(F_N,e_1)^{-1} \\
    &= Q_{1,N}^* A Q_{1,N}  + Q_{2,N}^* H_N Q_{2,N} + \beta^{-1} Q_{1,N}^* G_N K_N(F_N,e_1)^{-1},
\end{aligned}
\end{equation*}
where $G_N = \begin{bmatrix} g(0) & g'(0) & \ldots & g^{(N-1)}(0) \end{bmatrix}$.
Thus, for the norms of the products $F_N^\ell e_1$, $1 \leq \ell \leq N$, we get the following recursion:
\begin{equation*}
 \begin{aligned}
  \norm{F_N^\ell e_1} &\leq \norm{ \left( Q_{1,N}^* A Q_{1,N} + Q_{2,N}^* H_N Q_{2,N} \right) F_N^{\ell-1} e_1 } \\
  & \quad + \beta^{-1} \norm{ Q_{1,N}^* G_N K_N(F_N,e_1)^{-1} F_N^{\ell - 1} e_1 }\\
  & \leq \max(\norm{A},\norm{H_N}) \norm{F_N^{\ell-1} e_1} + \beta^{-1}  \norm{Q_{1,N}^* g^{(\ell - 1)}(0)} \\
  & \leq \norm{A} \norm{F_N^{\ell-1} e_1} +  \beta^{-1} c \, \norm{A}^{\ell - 1},
  %& \leq \norm{A} \norm{F_N^{\ell-1} e_1} + c \, \norm{A}^{\ell + 1},
  \end{aligned}
\end{equation*}
since $\beta > 1$ and $K_N(F_N,e_1)^{-1} F_N^{\ell - 1} e_1 = e_\ell$ for $1 \leq \ell \leq N$.
By induction we have that 
%Solving this recursion and bounding gives
$$
\norm{F_N^\ell e_1} \leq \norm{A}^\ell \norm{F_N^0 e_1} +\beta^{-1} c \, \ell \norm{A}^{\ell - 1} 
= \norm{A}^\ell + \beta^{-1} c \, \ell \norm{A}^{\ell - 1}
\leq \norm{A}^\ell ( 1 +\beta^{-1} c \, \ell).
$$
%We see that $\norm{F_N e_1} \leq 2 c \norm{A}$, and that the recursion above gives the statement. 
\end{proof}
\end{lem}

We are ready to give the following result, which gives
sufficient conditions for the convergence of the Arnoldi error.

\begin{thm}[Arnoldi error]\label{thm:arnoldierror}
Suppose there exists a constant $c>0$ such that \eqref{eq:w_l_assumption} is satisfied. Suppose
Algorithm~\ref{alg:infarn}
generates a Hessenberg matrix $F_N$ such that for some constant $C$, 
%$\|(tF_N)^N\|/N!\rightarrow 0$
$\|F_N^N\| \leq C^N$ %as $N\rightarrow\infty$.
for all $N>0$. Then, $b_N$ given by \eqref{eq:bNdef} satisfies
\[
%norm{r_N^0(t F_N)e_1} \rightarrow 0\;\;\textrm{ as }N\rightarrow\infty.
\norm{b_N} \rightarrow 0\;\;\textrm{ as }N\rightarrow\infty.
\]
Moreover, the Arnoldi error in \eqref{eq:error_expression}
satisfies
\[
\norm{u_N(t)-u_N^{IA}(t)}\rightarrow 0\;\;\textrm{ as }N\rightarrow\infty.
\]
\begin{proof}
We see from \eqref{eq:rN} that
\begin{equation} \label{eq:second_term_representation}
r_N(t F_N) e_1 = \sum\limits_{\ell=N}^\infty \frac{(t F_N)^\ell e_1}{\ell !} = \sum\limits_{k=1}^\infty (t F_N)^{Nk} 
\left( \sum\limits_{\ell=0}^{N-1} \frac{(tF_N)^\ell e_1}{ ( k N + \ell)!} \right).
\end{equation}
%Assuming there exists a constant $c$ such that \eqref{eq:g_growth} is satisfied, 
We see by Lemma~\ref{lem:powers_bound} that
\begin{equation} \label{eq:part_bound}
 \begin{aligned} 
\norm{\sum\limits_{\ell=0}^{N-1} \frac{(t F_N)^\ell e_1}{(kN+\ell)!} } & \leq \sum\limits_{\ell=0}^{N-1} \frac{(1+ c \, \ell) \norm{t A}^\ell}{ (kN + \ell)!} 
\leq (1 + \beta^{-1} c \, N) \sum\limits_{\ell=0}^{N-1} \frac{\norm{t A}^\ell }{ (kN + \ell)!} \\
&\leq (1 + \beta^{-1} c \, N ) \varphi_{kN}(t \norm{A}) \leq (1 + \beta^{-1} c  \, N )  \frac{\ee^{t \norm{t A}}}{ (kN)! }.
\end{aligned}
\end{equation}
In the last inequality above we use Lemma~\ref{lem:remainder_bound}. Thus, we see from \eqref{eq:second_term_representation} and 
\eqref{eq:part_bound} that
\begin{equation*}
\begin{aligned}
\norm{b_N} \leq \norm{Q_N} & \norm{r_N(t F_N) e_1} \beta \leq \sum\limits_{k=1}^\infty \left( \norm{(t F_N)^N}^k  \norm{\sum\limits_{\ell=0}^{N-1} \frac{(t F_N)^\ell e_1}{(kN+\ell)!} } \right) \beta \\
\leq & (\beta + c  \, N ) \sum\limits_{k=1}^\infty   \frac{\ee^{t \norm{A} }\norm{(t F_N)^N}^k }{ (kN)! } \beta
\leq   (\beta + c  \, N ) \ee^{t \norm{ A}}  \sum\limits_{k=1}^\infty   \frac{(tC)^{kN}  }{ (kN)! } \beta
\end{aligned}
\end{equation*}
which converges to zero as $N \rightarrow \infty$.
\end{proof}
\end{thm}

\begin{remark}[Assumptions in Theorem~\ref{thm:arnoldierror}]\label{rem:convcondition}\rm
Theorem~\ref{thm:arnoldierror} is only applicable when there exists a constant $C$
%when $\|(tF_N)^N\|/N!\rightarrow 0$, where $F_N$ is the Hessenberg
such that $\|F_N^N\| \leq C^N$ for all $N>0$, where $F_N$ is the Hessenberg
matrix generated by Algorithm~\ref{alg:infarn}.
This is a restriction on the generality of our convergence theory. 
In our numerical experiments we have seen no indication that
the assumption should not be satisfied (see Figure~\ref{fig:assumption_indicator}).
Moreover, the assumption can be motivated by certain intuitive 
uniformity assumptions and the generic behavior 
of Arnoldi's method for eigenvalue problems, as follows.
From the definition of the spectral radius, we have
\begin{equation}\label{eq:arnoldierr_uniform1}
  \|F_N^\ell\|^{1/\ell}\rightarrow \rho(F_N)\;\;\textrm{ as }\ell\rightarrow\infty.
\end{equation}
Moreover, under the condition that the Arnoldi method approximates the largest
eigenvalue of $A_\infty$, we also have
\begin{equation}\label{eq:arnoldierr_uniform2}
  \rho(F_N)\rightarrow\rho(A_\infty)\;\;\textrm{ as }N\rightarrow\infty.
\end{equation}
The operator $\rho(A_\infty)$ is block diagonal and the (1,1)-block is a finite operator $A$ and 
the (2,2)-block is a bounded operator (by assumption \eqref{eq:HNasm}). 
Hence, it is natural to assume that $\rho(A_\infty)=d\in\RR$ exists.
If $\rho(A_\infty)$ exists and 
the limits \eqref{eq:arnoldierr_uniform1} and 
\eqref{eq:arnoldierr_uniform2} hold also in a uniform sense, we have
that $\|F_N^N\|^{1/N}\rightarrow d$, which implies the assumption.

%, which implies that $\|(tF_N)^N\|/N!\rightarrow 0$. 
%Since $A_\infty$ is block triangular, the eigenvalues
%are the union of the eigenvalues of $A$ and the eigenvalues of $H_\infty$. 
%Suppose the largest eigenvalue of $A_\infty$
%is an eigenvalue of $A$.
%
%According to the theory 
%of Arnoldi's method for eigenvalue problems, 
%the extreme isolated eigenvalues are often well approximated. 
%Suppose this coincides with the eigenvalues largest in modulus, and
%suppose the largest eigenvalues in modulus of $F_N$,
%converges to the largest eigenvalue of $A_\infty$.
% Under these assumptions 
%\begin{equation} \label{eq:spectral_rad_conv}
%\rho(F_N) \rightarrow \rho(A),
%\end{equation}
%where $\rho$ denotes the spectral radius.
%%:  $\rho(A) = \underset{\lambda \in \Lambda(A)}{\max} \abs{\lambda}$.
%Gelfand's formula states that for every matrix $A$ and every matrix norm $\norm{ \cdot }$,
%$$
%\norm{A^\ell}^{1/\ell} \rightarrow \rho(A), \quad \textrm{as} \quad \ell \rightarrow \infty.
%$$
%Based on the assumption \eqref{eq:spectral_rad_conv} and Gelfand's formula, we further assume that
%\begin{equation} \label{eq:F_N_assumption}
%\norm{F_N^N}^{1/N} \rightarrow \rho(A_\infty), \quad \textrm{as} \quad N \rightarrow \infty.
%\end{equation}
%%The validity of this assumption will be illustrated by numerical experiment
%%Moreover, we will make a mild assumption on the growth of the vectors $w_\ell$.
%%\end{proof}
%%Thus the first term of \eqref{eq:err_bound_1} goes to zero as $N \rightarrow \infty$. 
\end{remark}

%%%%%%%%%%%%%%%%%%%%%%%%%%%%%%%%%%%%%%%%%%%%%%%%%%%%%%%%%%%%%%%%%%%%%%%%%%%%%%%%%%%%%%%%%%%%%%
\section{Numerical examples}\label{sect:examples}
%%%%%%%%%%%%%%%%%%%%%%%%%%%%%%%%%%%%%%%%%%%%%%%%%%%%%%%%%%%%%%%%%%%%%%%%%%%%%%%%%%%%%%%%%%%%%%

%We first mention some computational methods to evaluate the
%derivatives $g^{(\ell)}(0)$ which are needed to obtain the coefficient vectors
%$w_k$ of the expansion \eqref{eq:expansion}.

%Then, to illustrate the behavior of the infinite Arnoldi algorithm, 
%we perform numerical comparisons between
%the three different choices of the basis functions discussed above: the monomial basis,
%the Bessel and the modified Bessel functions of the first kind.
%
%In order to illustrate the effect of the choice of the basis, we show the convergence
%of the algorithms for several problems.  
%To show the potential of the infinite Arnoldi algorithm, we also make some comparisons on efficiencies
%between the new methods and some well established codes.

\subsection{Numerical evaluation of the derivatives $g^{(\ell)}(0)$}
In order to carry out $N$ steps of the algorithm,
we need the expansion coefficients $w_0,\ldots,w_N$, 
which are directly available from the derivatives $g^{(\ell)}(0)$, $\ell=0,\ldots,N$ via \eqref{eq:Bessel_Krylov_relation}. 
%To evaluate the derivates $g^{(\ell)}(0)$
If the nonlinearity is not explicit such it is
not possible to compute expressions for the
derivatives by hand,  there are several alternatives.
One may use, e.g., symbolic differentiation which
is available for several special functions in MATLAB. 
or the techniques of \it automatic differentiation \rm can be used \cite{Griewank}.

Another alternative is to use matrix functions. If an efficient and numerically stable  matrix function implementation of 
$h(z)$ is available (see, e.g., \cite[Ch.\;4]{Higham}), one may use the fact that
% $$
% g(H)  = 
%  \begin{bmatrix}
%    g(0) & & & \\
%    g'(0) & g(0) & & \\
%    g''(0)/2 & \ddots & \ddots & & \\
%    \vdots & &  &  & \\
%    g^{(N-1)}(0)/N! & &  \ldots & g'(0) & g(0) \\
%  \end{bmatrix}
% \,\,
%  \textrm{for} \,\,
% H = \begin{bmatrix}
%    0 & & &  \\
%    1 &0 & &  \\
%  %   &1 &0 & & & \\
%     &\ddots & \ddots &   \\
%     & & 1 & 0  
%  \end{bmatrix}  \in \mathbb{R}^{N \times N}.
% $$
$$\small
h(H) \, e_1  = 
 \begin{bmatrix}
   h(0)  \\
   h'(0) \\
   h''(0)/2  \\
   \vdots \\
   h^{(N-1)}(0)/N! \\
 \end{bmatrix}
\quad
 \textrm{for} \quad
H = \begin{bmatrix}
   0 & & &  \\
   1 &0 & &  \\
 %   &1 &0 & & & \\
    &\ddots & \ddots &   \\
    & & 1 & 0  
 \end{bmatrix}  \in \mathbb{R}^{N \times N}.
$$
Also, there exists methods
to compute derivatives by numerically integrating
the contour integral in the Cauchy integral formula  \cite{Bornemann:2011:CAUCHY}.
% $$
% g^{(N)}(0) = \frac{N!}{2 \pi \ii} \int_\Gamma \frac{g(\lambda)}{\lambda^{N+1}} \, \dd \lambda.
% $$
% Choosing $\Gamma$ to be a circle of radius $r$ and using the trapezoidal 
% rule and additional information about the entire function $g$ (e.g., type and order), one may
% choose $r$ and number of nodes in a way that minimizes the error.
% By using the algorithm on \cite[pp.\;13]{Bornemann:2011:CAUCHY} and the formula (8.10) of \cite{Bornemann:2011:CAUCHY} 
% to choose $r$, we get the errors depicted in Figure~\ref{fig:approximations} for the 80 first derivatives $g^{(\ell)}(0)$
% of the Airy function $\mathrm{Ai}(x)$. Compared to symbolic differentiation of Matlab, this approach was 10 times faster.

%%%%%%%%%%%%%%%%%%%%%%%%%%%%%%%%%%%%%%%%%%%%%%%%%%%%%%%%%%%%%%%%%%%%%%%%%%%%%%%%%%%%%%%%%%%%%%
\subsection{1-D Schr\"odinger equation with inhomogeneity} \label{sec:schrode1}
%%%%%%%%%%%%%%%%%%%%%%%%%%%%%%%%%%%%%%%%%%%%%%%%%%%%%%%%%%%%%%%%%%%%%%%%%%%%%%%%%%%%%%%%%%%%%%

%\subsubsection{1D case}

We first consider a finite difference spatial discretization (with 100 points) of the initial value problem
\begin{equation} \label{eq:schrode_ivp1}
\ii \partial_t u =  - \epsilon \partial_{xx}u + f(t)  \sin(2^4 \pi x(1-x)),  \quad x \in [0,1], \quad t \in [0,T] 
\end{equation}
%\begin{equation} \label{eq:schrode_ivp1}
%\begin{aligned}
%\ii \partial_t u &=  - \epsilon \partial_{xx}u + f(t)  \sin(2^4 \pi x(1-x)),  \quad x \in [0,1], \quad t \in [0,T] \\
%u(x,0) &= \exp(-100(x-0.5)^2) \\
%\end{aligned}
%\end{equation}
%
subject to periodic boundary conditions, with $f(t) = (1+ \ii) \sin(t)^2$ and initial condition $u(x,0) = \exp(-100(x-0.5)^2)$.
%\begin{equation} \label{eq:f_t}
%f(t) = (1+ \ii) \sin(t)^2.
%\end{equation}
Figure~\ref{fig:compare_expansion_sin2} depicts the absolute error of the approximation $f(t) \approx W_N \exp(t H_N) e_1 = \sum\limits_{\ell=0}^{N-1} w_\ell \phi_\ell^{(N)}(t)$,
for the three different choices of $W_N$ and $H_N$, for $1 \leq N \leq 50$ and $t=6$.
%\begin{figure}[h!]
%\begin{center}
%\includegraphics[scale=0.45]{approx_sin2.eps}\,
%\end{center}
%\caption{The absolute error vs. expansion size $N$ for the approximation of $f(t) = \sin^2(t)$ for
%the three different choices of basis functions.}
%\label{fig:compare_expansion_sin2}
%\end{figure}
%
We again compare the infinite Arnoldi algorithm approximation of $u(T)$ for the three different
expansion of $f(t)$. In Figures~\ref{fig:schrode_1d_convergence} we illustrate the relative 2-norm error
of the approximations vs. the Krylov subspace size, when $\epsilon = 10^{-5}$ and $\epsilon = 10^{-3}$. 
%The effect of the inhomogeneity as seen in Figure~\ref{fig:compare_expansion_sin2}
%can be seen in case of the weak linear part ($\epsilon = 10^{-5}$, $T=10$).
In Fig.~\ref{fig:compare_expansion_sin2} we observe different truncation errors
for different basis functions. Analogously, a difference in convergence speed of Alg.~\ref{alg:infarn}
can be observed in Fig.~\ref{fig:1a}.

In a sense, the convergence of the linear part (associated with $A$) dominates the total error
in the case of the strong linear part ($\epsilon = 10^{-3}$),
and therefore the choice of basis does not affect the convergence,
which is also observed in Fig.~\ref{fig:1b}.

% In the following, we set $\epsilon = 10^{-4}$, and integrate up to $t=2$.
% We see that as $\epsilon$ grows, the linear part takes over and the difference in the convergences
% of the expansions of $g(t)$ becomes less visible.

% \begin{figure}[h!]
% \begin{center}
% \includegraphics[scale=0.42]{schrod_sin2_2.eps} \hspace{-6mm}
% \includegraphics[scale=0.42]{schrod_sin2_3.eps}
% \end{center}
% \caption{Error vs. the Krylov subspace size for the Schr\"odinger example. Left: $\epsilon = 10^{-5}$, $T = 10$,
% right: $\epsilon = 10^{-3}$, $T = 0.5$.}
% \label{fig:schrode_1d_convergence}
% \end{figure}

\begin{figure}
\begin{subfigure}{0.3\textwidth}
\hspace{-1.2cm}\includegraphics[scale=0.32]{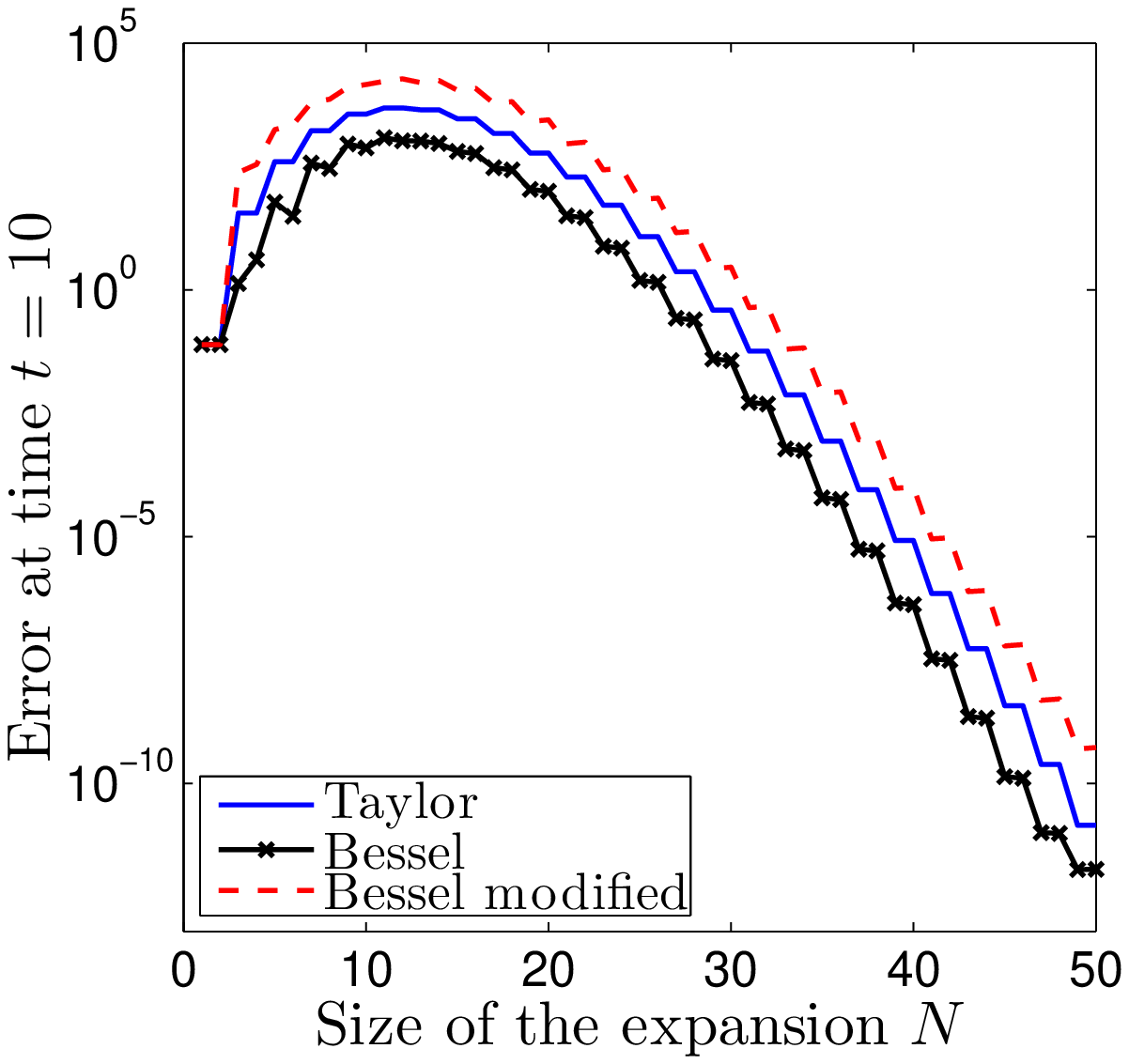}
\caption{}
\label{fig:compare_expansion_sin2}
\end{subfigure}
\begin{subfigure}{0.3\textwidth}
\includegraphics[scale=0.32]{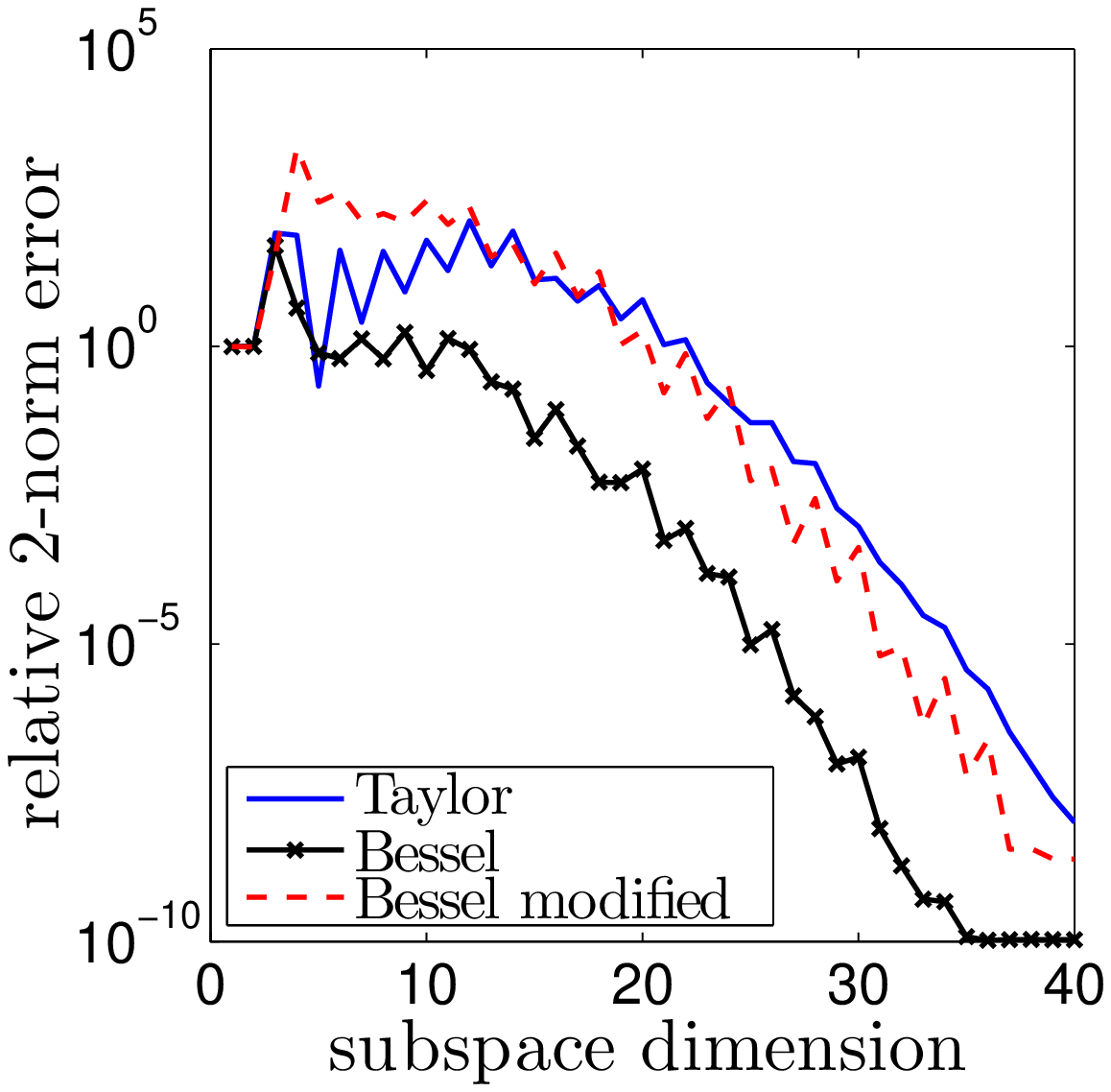}
\caption{ \hspace{-15mm} } \label{fig:1a}
\end{subfigure}
\hspace*{\fill} % separation between the subfigures
\begin{subfigure}{0.3\textwidth}
\includegraphics[scale=0.32]{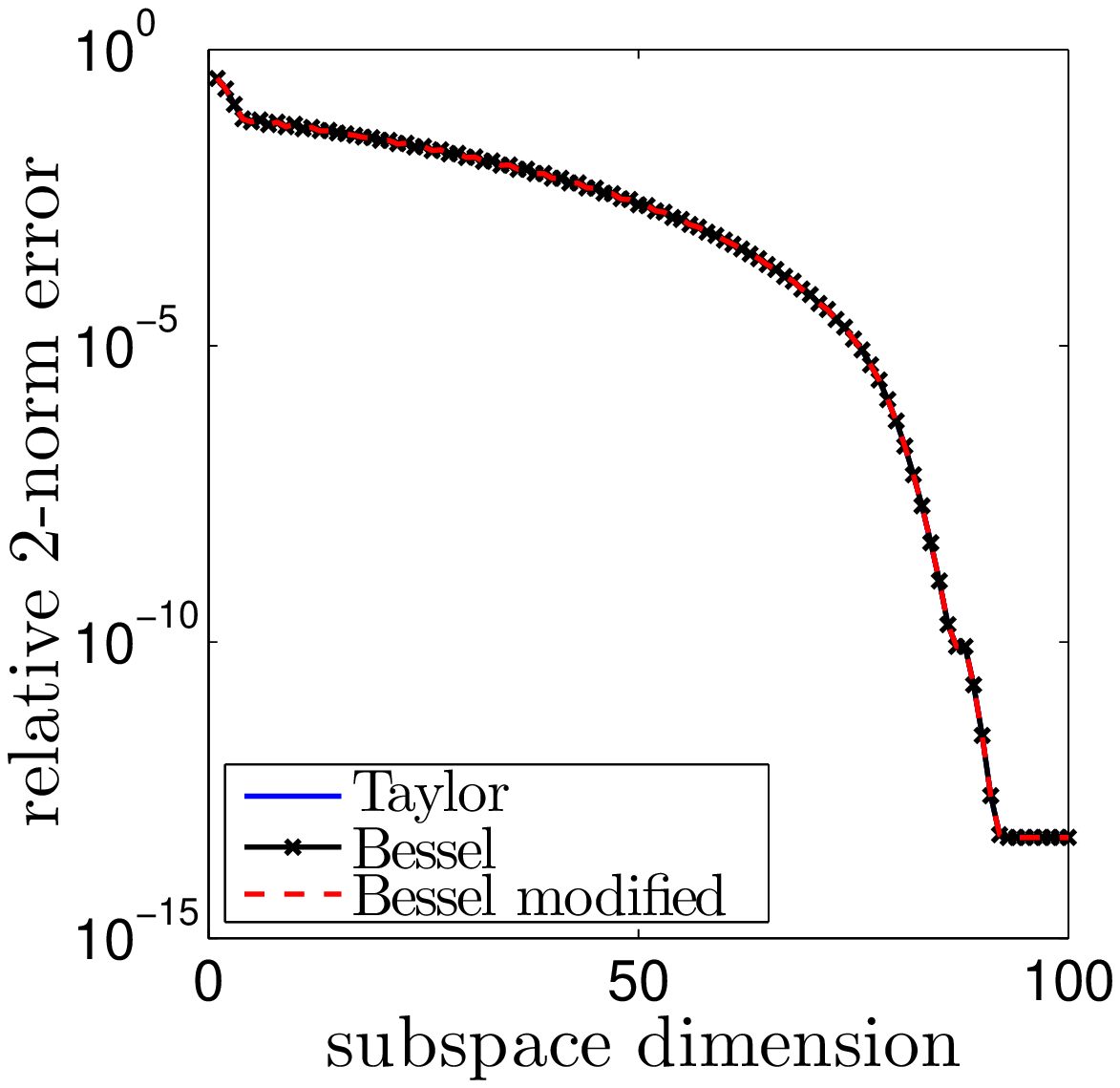}
\caption{ \hspace{-15mm} } \label{fig:1b}
\end{subfigure}
\vspace{-5mm}
\caption{
Subfigure (a) shows the absolute error vs. expansion size $N$ for the approximation of $f(t) = \sin^2(t)$ for the three different choices of basis functions.
Subfigure (b) and (c) shows the error vs. the Krylov subspace size for the Schr\"odinger example. (b) $\epsilon = 10^{-5}$, $T = 10$,
(c) $\epsilon = 10^{-3}$, $T = 0.5$.
} \label{fig:schrode_1d_convergence}
\end{figure}

We illustrate the competitiveness of the approach in 
terms of CPU-time\footnote{All experiments are carried out on a desktop computer with a 2.90 GHz single Pentium processor using MATLAB.}
 in Figures~\ref{fig:schrode_1d_timing}, 
when $\epsilon = 10^{-5}$ and $\epsilon = 10^{-3}$. We use three different integrators: 
the infinite Arnoldi algorithm with the Bessel functions of the first kind and the MATLAB implementations of the Runge-Kutta method \verb|ode45|  and
\verb|ode15s|.

Note that the Matlab integrators use adaptive time-stepping, and that the infinite
Arnoldi method performs a single time step for which the subspace size is set a priori.  
When $\epsilon = 10^{-5}$, \verb|ode45| needed 10,16,25,40,86 time steps to obtain the
results of Figure~\ref{fig:schrode_1d_timing}, and \verb|ode15s| 10,13,51,96,189, respectively.
When $\epsilon = 10^{-3}$, \verb|ode45| needed 29,30,31,33,33 time steps, 
and \verb|ode15s| 10,15,22,60,124 time steps.

When the linear part is not very stiff, we see that the explicit integrator \verb|ode45| gives better results than
the stiff implicit solver \verb|ode23|. For this particular simulation setup,
the infinite Arnoldi method is faster than the MATLAB Runge-Kutta implementations, as can be observed in 
Figures~\ref{fig:schrode_1d_timing}.

Figure~\ref{fig:assumption_indicator} gives a numerical 
justification for the assumptions used in
the error analysis given in Section~\ref{sec:krylov_error}. We consider the numerical example
above with the parameter $\epsilon = 10^{-3}$. We observe
 that up to machine precision,
$\norm{F_N^N}^{1/N} \rightarrow \rho(A)$ as $N \rightarrow \infty$, such that the conditions discussed
in Remark~\ref{rem:convcondition} appear to be satisfied.

\begin{figure}[h!]
\begin{subfigure}{0.3\textwidth}
\hspace{-1.4cm}\includegraphics[scale=0.3]{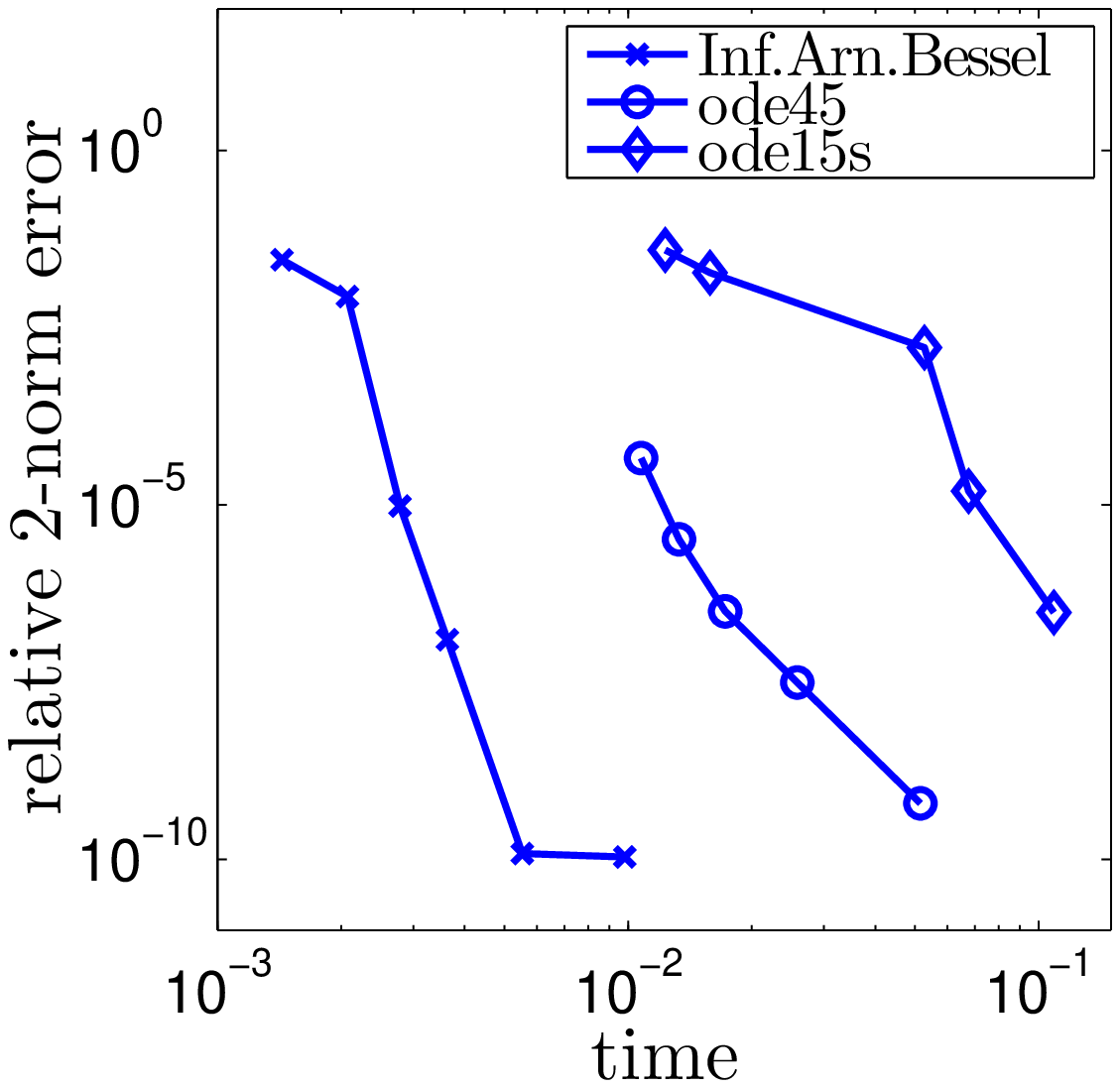}%
\caption{}
\end{subfigure}
\begin{subfigure}{0.3\textwidth}
\hspace{-1.0cm}\includegraphics[scale=0.3]{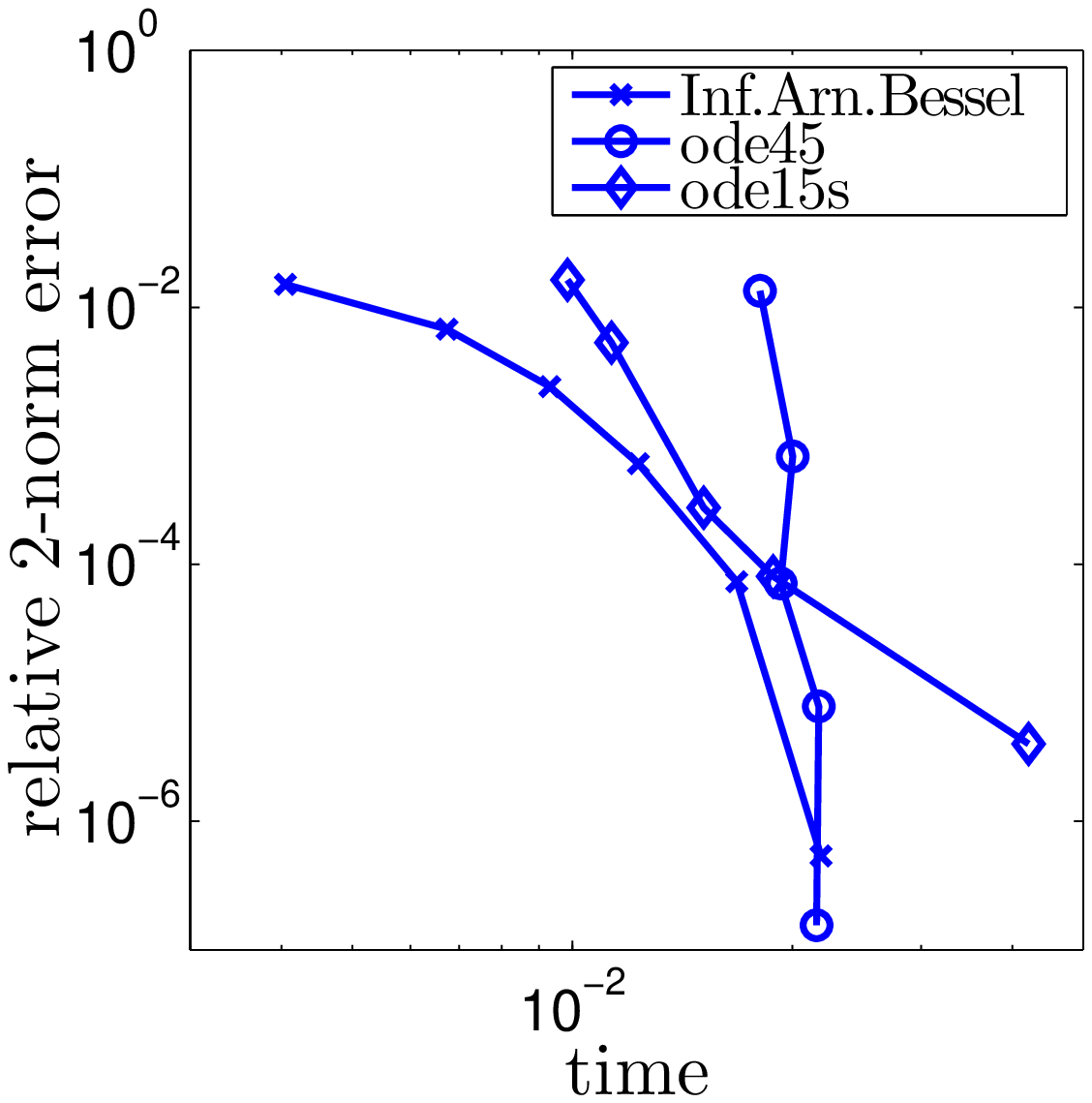}%
\caption{}
\end{subfigure}
\begin{subfigure}{0.3\textwidth}
\hspace{-0.5cm}\includegraphics[scale=0.3]{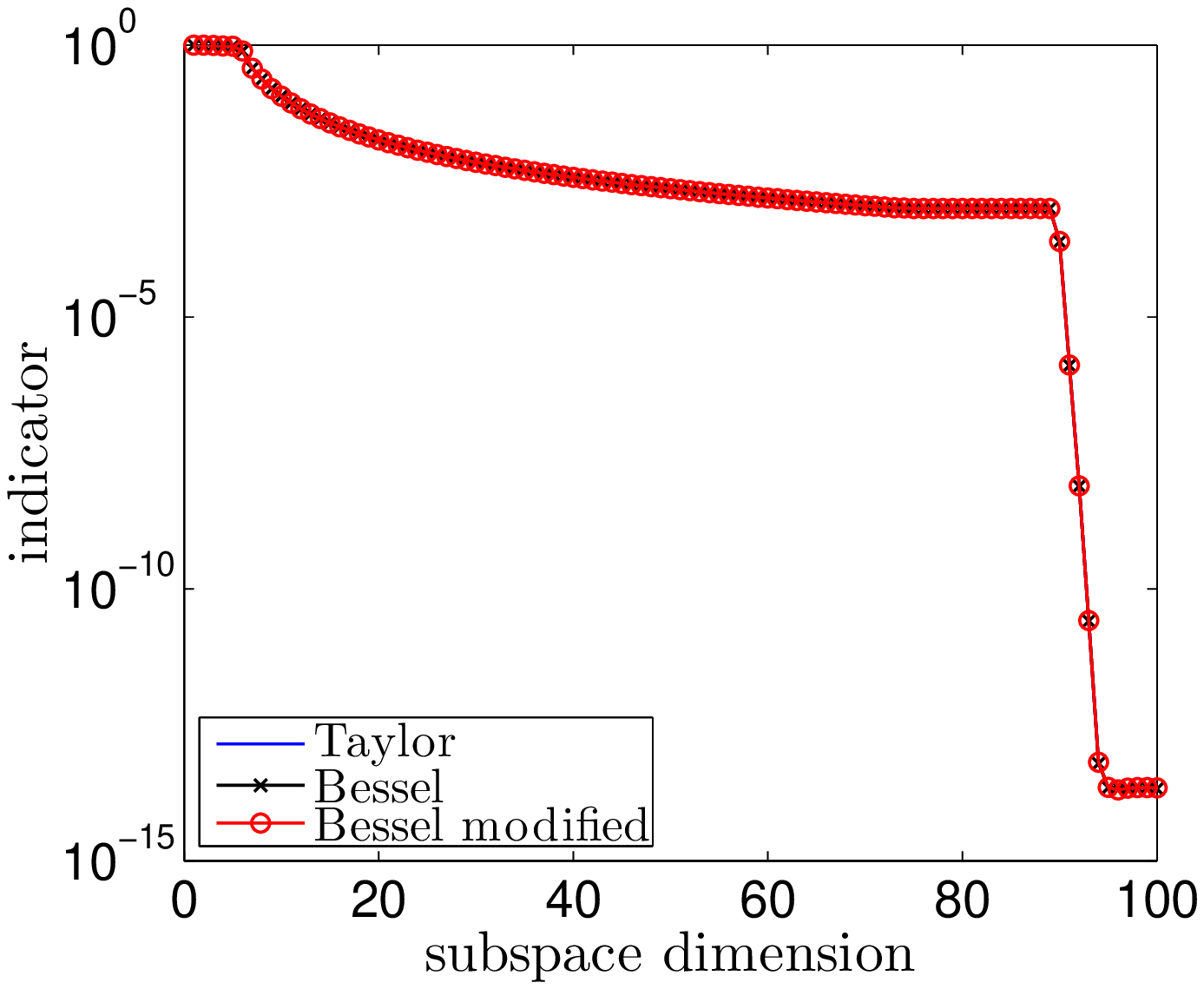}%
\caption{}\label{fig:assumption_indicator}
\end{subfigure}
%\hspace{-6mm}
%\subfigure{\columnwidth}{\includegraphics[scale=0.42]{time_comparison_2.eps}}
%\subfigure[]{a%\includegraphics[scale=0.42]{assumption_error.eps}%\label{fig:assumption_indicator}
%}
\caption{Subfigures (a) and (b) show 
the error vs. CPU time in seconds for the 1-D Schr\"odinger example, with (a) $\epsilon = 10^{-5}$, $T = 10$, and
(b) $\epsilon = 10^{-3}$, $T = 0.5$. Subfigure (c)
show the indicator 
$\abs{ \norm{F_N^N}^{1/N} / \rho(A) - 1}$}
\label{fig:schrode_1d_timing}
\end{figure}

%\begin{figure}[h!]
%\begin{center}
%\includegraphics[scale=0.45]{assumption_error.eps}\,
%\end{center}
%\caption{The indicator $\abs{ \norm{F_N^N}^{1/N} / \rho(A) - 1}$.}
%\label{fig:assumption_indicator}
%\end{figure}
%

\subsection{2-D Schr\"odinger equation with inhomogeneity} %\label{sec:schrode2}

In order to illustrate generality of the infinite Arnoldi method, 
we consider a finite difference spatial discretization (with $100^2$ points) of the  two-dimensional initial value problem
\begin{equation} \label{eq:schrode_ivp2}
\ii \partial_t u =  - \epsilon (\partial_{xx}u + \partial_{yy}u) + 
f(t)  \sin(2^4 \pi x(1-x)y(1-y)),  \quad x \in [0,1], \quad t \in [0,T] 
%\end{aligned}
\end{equation}
subject to periodic boundary conditions, with $f(t)$ as in \eqref{eq:schrode_ivp1} and
initial condition $u(x,0) = \exp(-100\big((x-0.5)^2 + (y-0.5)^2))$

%Numerical implementation is performed in real arithmetics.
%, i.e., we discretize the function $v(t) = \begin{bmatrix} \textrm{Re} 
%\,\, u(t) & \textrm{Im} \,\, u(t) \end{bmatrix}^\mathrm{T}$. 
%Therefore, the spatial semidiscretization gives an ODE of dimension 200.

We compare the infinite Arnoldi algorithm approximation of $u(T)$ for the three different
expansion of $f(t)$. Figures~\ref{fig:schrode_2d_convergence} depict the relative 2-norm error
of the approximations vs. the Krylov subspace size, when 
$\epsilon = 5 \cdot 10^{-3}$ and $\epsilon = 5 \cdot 10^{-2}$. 
We see again that the convergence of the linear part starts to dominate the total error
as the linear part gets larger.

% In the following, we set $\epsilon = 10^{-4}$, and integrate up to $t=2$.
% We see that as $\epsilon$ grows, the linear part takes over and the difference in the convergences
% of the expansions of $g(t)$ becomes less visible.

Figures~\ref{fig:schrode_2d_timing} depict the relative 2-norm
errors of the approximations of $u(t)$
vs. the CPU time when $\epsilon = 5 \cdot 10^{-2}$ and
$\epsilon = 5 \cdot 10^{-3}$, for the three different integrators: 
infinite Arnoldi with Bessel expansion and Matlab codes \verb|ode45| and \verb|ode15s|.  
%(Here about the settings of the methods... the methods use adaptive step sizes, the infinite
%Arnoldi is used by setting the subspace size a priori....). 
%As the linear part is not very stiff, the explicit integrator \verb ode45  gives better results than
%the stiff solver $\ode23$.

\begin{figure}[h!]
\begin{center}
\includegraphics[scale=0.32]{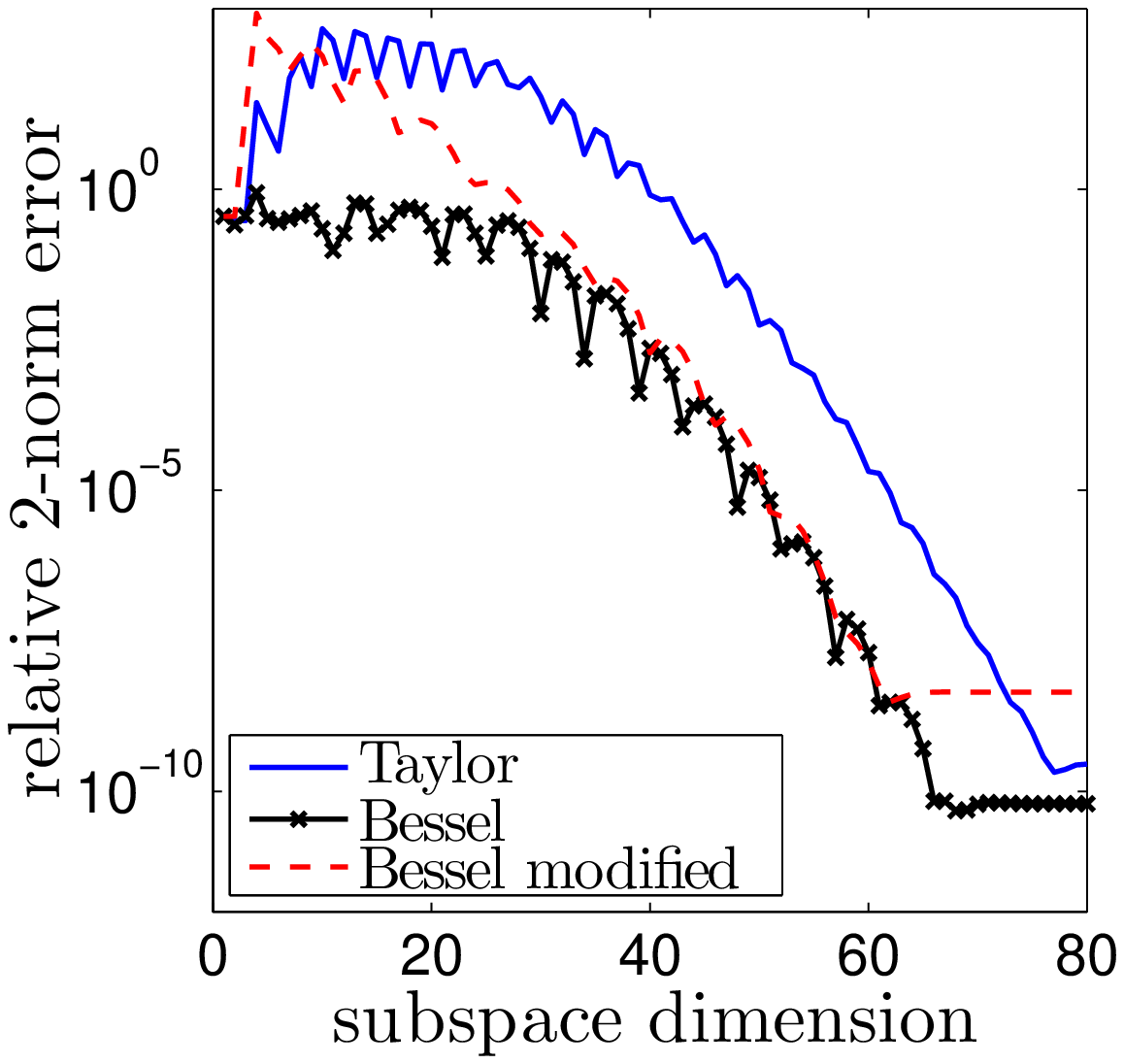}\hspace{-6mm}
\includegraphics[scale=0.32]{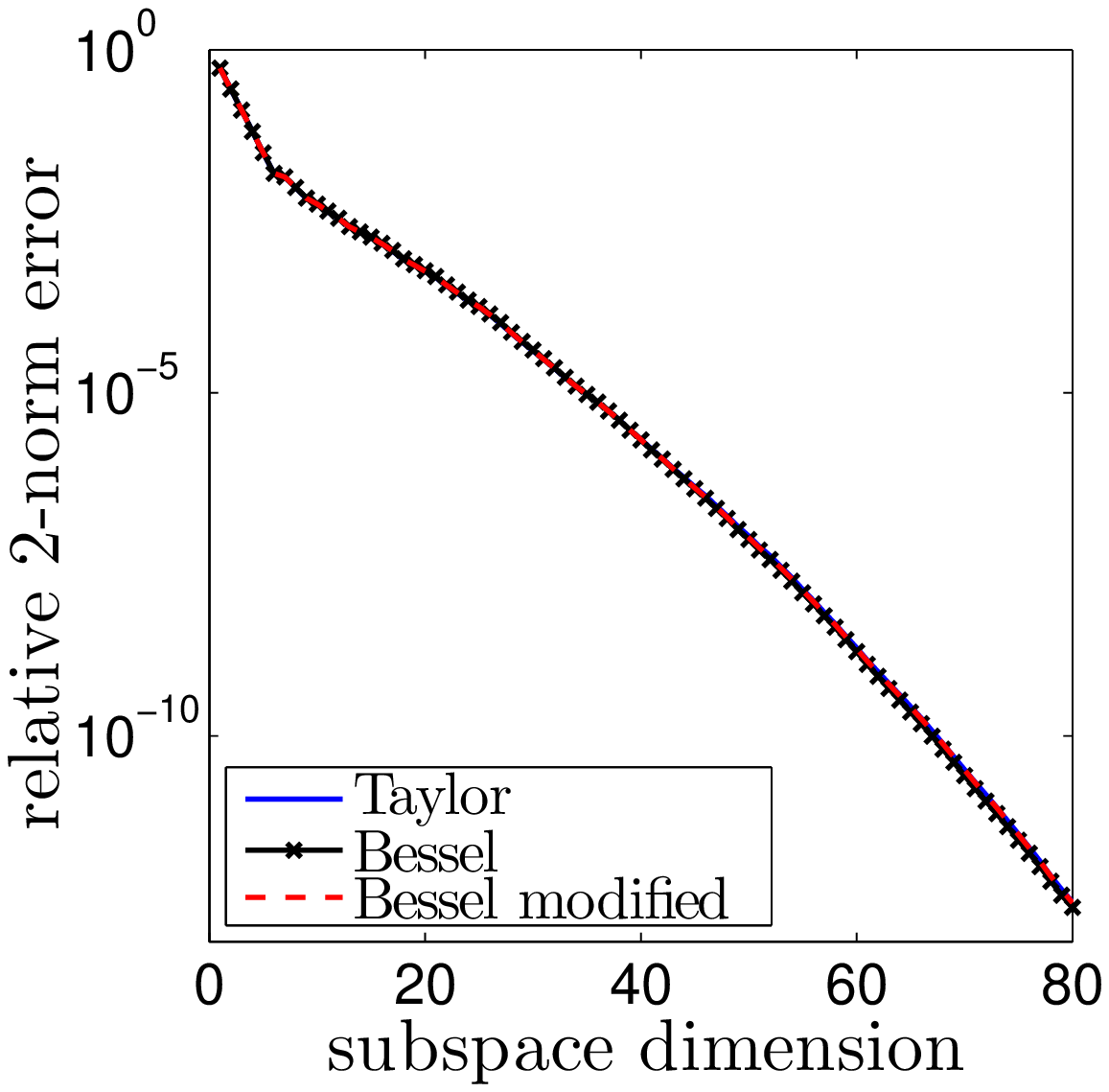}
\end{center}
\caption{Error vs. the Krylov subspace size for the Schr\"odinger example. Left: $\epsilon = 5 \cdot 10^{-3}$, $T = 10$,
right: $\epsilon = 5 \cdot 10^{-2}$, $T = 0.25$.}
\label{fig:schrode_2d_convergence}
\end{figure}

\begin{figure}[h!]
\begin{center}
\includegraphics[scale=0.32]{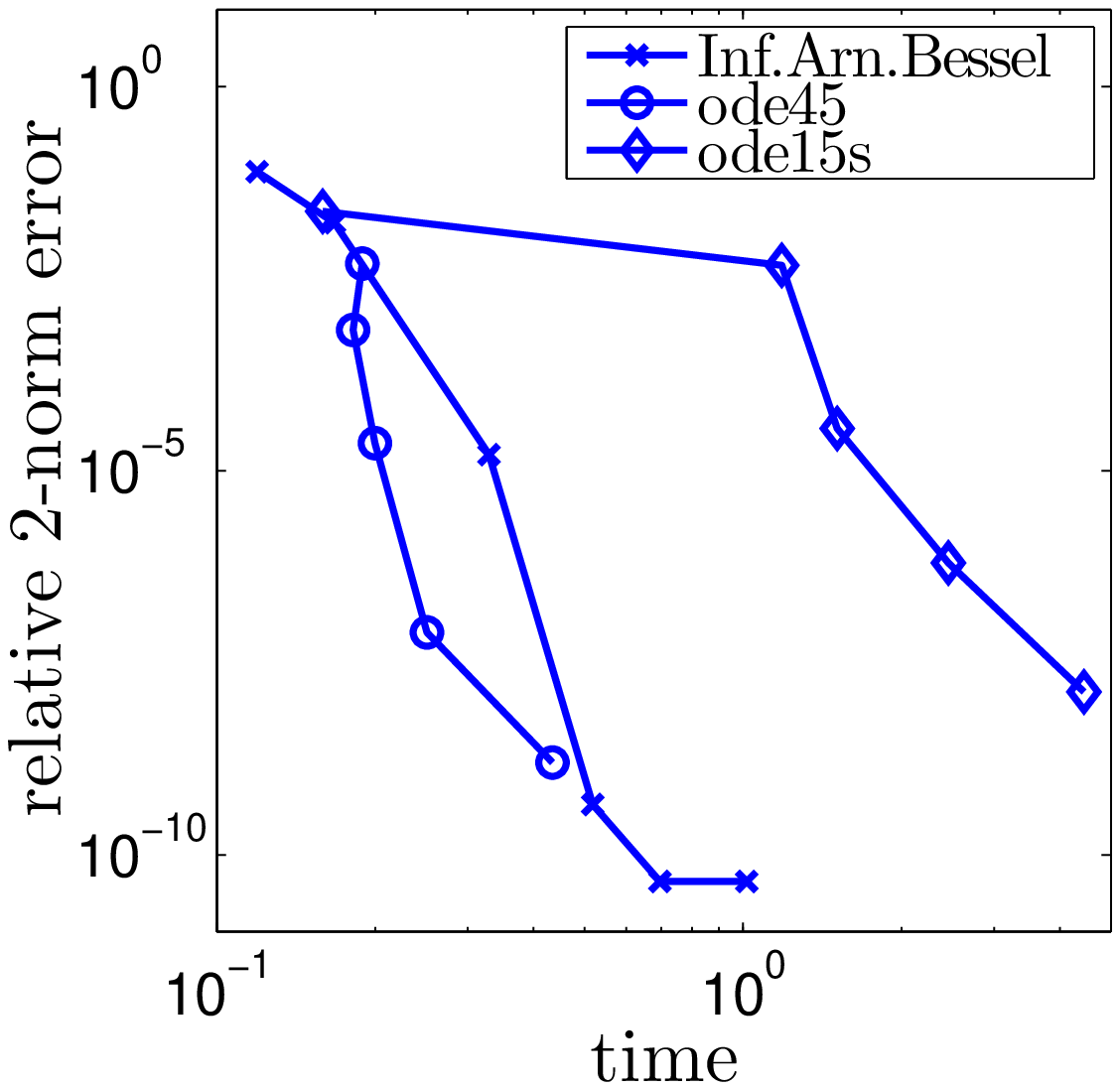}\hspace{-6mm}
\includegraphics[scale=0.32]{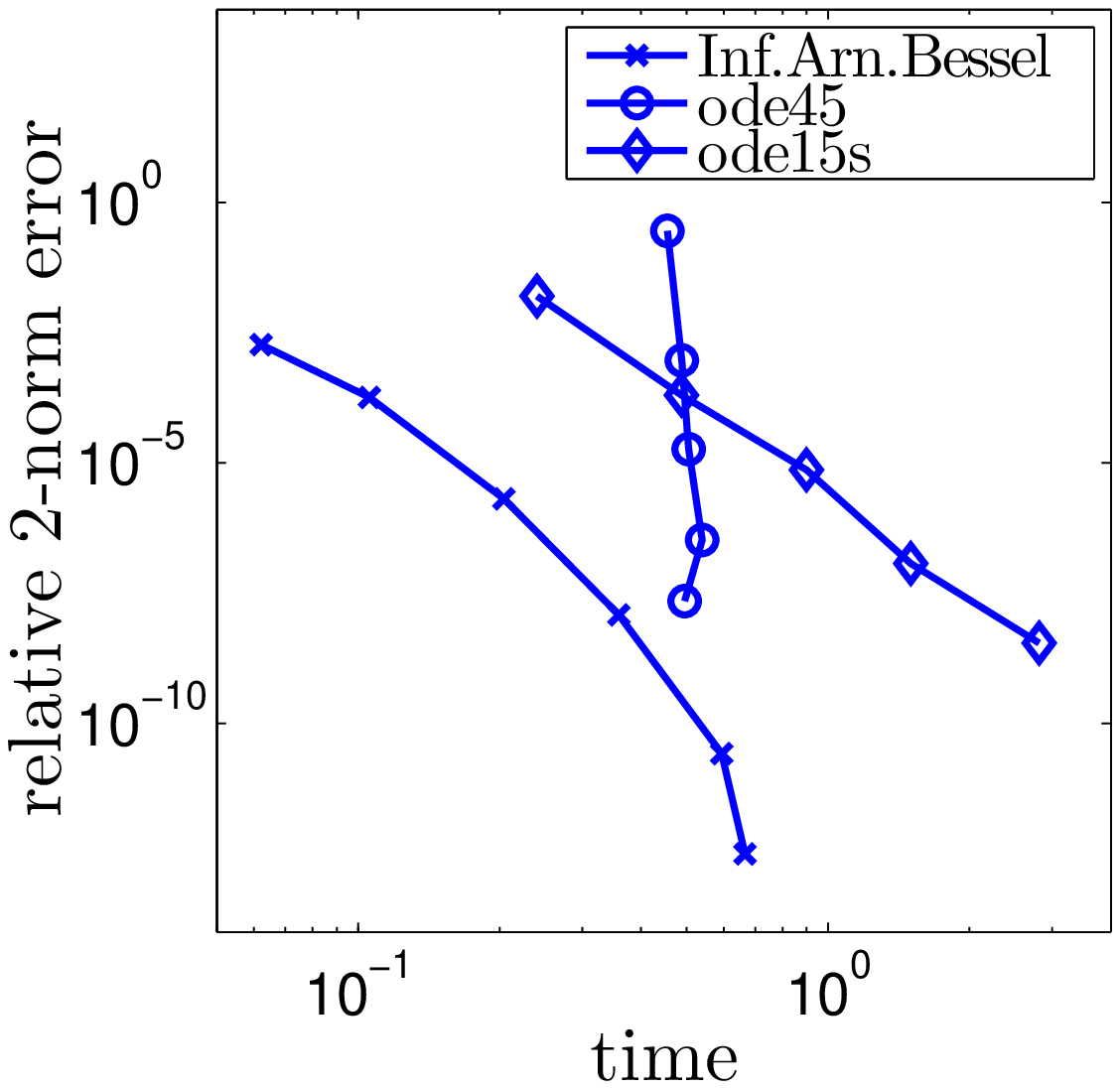}
\end{center}
\caption{Error vs. CPU time in seconds for the 2-D Schr\"odinger example. 
Left: $\epsilon = 5 \cdot 10^{-3}$, $T = 10$,
right: $\epsilon = 5 \cdot 10^{-2}$, $T = 0.25$.}
\label{fig:schrode_2d_timing}
\end{figure}

\section{Concluding remarks and outlook}
The main contribution of this paper
is a new algorithm for inhomogeneous linear ODEs and
the associated convergence theory. The algorithm
belongs to a class of methods \emph{exponential integrators}.
Many of the techniques that are combined with exponential 
integrators are likely feasible in this situation.
For instance, a potentially faster approach 
can be derived by repeating the algorithm for different $t$, i.e., 
instead of integrating to $t=T$ directly,
the algorithm can be applied for $h_1,\ldots,h_m$ where $T=h_1+\cdots+h_m$.
Moreover, it seems also feasible to apply the 
algorithm to certain nonlinear equations, by simple linearization
procedure, although it would certainly not be 
efficient for all nonlinear problems.
See \cite{Hochbruck:2010:EXPINT} for variants of exponential
integrators.

\section*{Acknowledgments}
The authors thank Stefan G\"uttel for several valuable discussions.

\bibliographystyle{plain}
\bibliography{eliasbib,misc}

\appendix

\section{Additional results needed in proofs}

\begin{lem}\label{lem:lemma_bessel_exp_1}
Let $H_N \in \RR^{N \times N}$ be defined as in \eqref{eq:H_bessel}
or as in \eqref{eq:H_bessel_modified},
and let the eigendecomposition of $H_N$ be given as $H_N = V \Lambda V^{-1}$.
Then, the condition number of the eigenvector matrix in 2-norm, i.e., $\kappa_2(V) = \norm{V} \norm{V^{-1}}$,
is given by
$\kappa_2(V) = \sqrt{2}$.
Moreover, 
\begin{equation} \label{eq:Bessel_exp_norm_bound}
\norm{ \ee^{t H_N} } \leq \sqrt{2} \, \ee^{t \alpha(H_N)},
\end{equation}
where $\alpha(A)$ denotes the spectral abscissa of $A$. If $H_N$ is given by 
\eqref{eq:H_bessel}, $\alpha(H_N)=0$ and if $H_N$ is given by  
\eqref{eq:H_bessel_modified} we have that $\alpha(H_N)\le 1$. 
\end{lem}
\begin{proof}
We first consider the case where $H_N$ is defined by \eqref{eq:H_bessel_modified}. Note that $H_N$ is the colleague matrix of the Chebyshev polynomial $T_N(x)$ \cite[Theorem~18.1]{Trefethen},
and we know that $H_N$ has $N$ different eigenpairs $(\lambda,v)$ where $\lambda$-values are the zeros of $T_N(x)$, and $v$-vectors are of the form $v = \begin{bmatrix} T_0(\lambda) & \hdots & T_{n-1}(\lambda) \end{bmatrix}^\mathrm{T}$.
Thus, $H_N$ has the eigendecomposition $H_N = V \Lambda V^{-1}$, where
$$
 V = \begin{bmatrix} T_0(t_0) & \ldots & T_0(t_{N-1}) \\ \vdots & & \vdots \\ T_{N-1}(t_0) & \ldots & T_{N-1}(t_{N-1}) \end{bmatrix},
$$
and $(t_0,\ldots,t_{N-1})$ are the $N$ different zeros of $T_N(\cdot)$.
The Chebyshev polynomials satisfy a discrete orthogonality condition 
%(see \cite{Trefethen})
\begin{equation} \label{eq:discrete_ortho}
\small \sum_{k=0}^{N-1}{T_i(t_k)T_j(t_k)} =
\begin{cases}
0 &, \quad i\ne j \\
N &, \quad i=j=0 \\
N/2 &, \quad i=j\ne 0.
\end{cases}
\end{equation}
With \eqref{eq:discrete_ortho} we verify that $VV^*$ is a diagonal matrix
where all elements are equal to $N/2$ except the first element which is equal to $N$. 
Hence, $R$-matrix in the QR-decomposition of $V^*=QR$ is a diagonal matrix
and we conclude that there exists $Q \in \RR^{n \times n}$ such that 
$QQ^*=Q^* Q = I$ and 
$V=\alpha \operatorname{diag}(\sqrt{2},1,\ldots,1)Q^*$,
%$$
%V = \alpha \begin{bmatrix} \sqrt{2} & & & \\ & 1 & & \\ & & \ddots & \\ & & & 1 \end{bmatrix} Q^*
%$$
where $\alpha=\sqrt{N/2}$.
We we see that $\norm{V} = \abs{\alpha} \sqrt{2}$, and $\norm{V^{-1}} = 1/\abs{\alpha}$. Therefore $\kappa_2(V) = \sqrt{2}$.

Let now $H_N$ be defined as in \eqref{eq:H_bessel}. Define the polynomials $\wt T_n(x)$, $n \geq 0$ as
$\wt T_n(x) = \ii^n T_n( - \ii x)$,
where $T_n$ is the $n$th Chebyshev polynomial. We now use
the recurrence relation
of Chebyshev polynomials; see, e.g., \cite[Chapter~3]{Trefethen}. 
We see that $\wt T_i$ satisfies
$\wt T_0(x) = 1$, $\wt T_1(x) = x$
and $\wt T_{n+1}(x) = -2x \wt T_n(x) + \wt T_{n-1}(x)$.
Therefore, the eigenvalues of $H_N$ are the zeros of the polynomial $\wt T_n(x)$, which
are $\ii$ multiplied with the zeros of the polynomial $T_n(x)$. The corresponding eigenvectors are of the form 
$v = \begin{bmatrix} \wt T_0(\lambda) & \ldots & \wt T_{n-1}(\lambda) \end{bmatrix}^\mathrm{T}$.
From the condition \eqref{eq:discrete_ortho} it follows that the polynomials $\wt T_i(x)$ satisfy the condition
\begin{equation*} %\label{eq:discrete_ortho}
\sum_{k=0}^{N-1}{T_i(t_k)T_j(t_k)} =
\begin{cases}
0 &, \quad i\ne j \\
N &, \quad i=j=0 \\
\ii^{i+j} \, N/2 &, \quad i=j\ne 0
\end{cases}
\end{equation*}
and the rest of the proof follows as for the modified Bessel functions.
The bound \eqref{eq:Bessel_exp_norm_bound} follows from the
fact that 
$\norm{ \ee^{t H_N} } = \norm{V  \ee^{t \Lambda} V^{-1}} \le \kappa(V)\|\ee^{t\Lambda}\|$.
The conclusion about the 
spectral abscissa if $H_N$ is given by
 \eqref{eq:H_bessel_modified} follows from Gershgorin's theorem and
the conclusion of if $H_N$ is given by
\eqref{eq:H_bessel} follows from the fact that the eigenvalues of $H_N$ are imaginary.
\end{proof}
%
%Corollary: the exponential of $H_N$ is bounded as follows
%\begin{equation} \label{eq:Bessel_exp_norm_bound}
%\norm{ \ee^{t H_N} }_2 = \norm{V  \ee^{t \Lambda} V^{-1}}_2 \leq \sqrt{2} \, \ee^{\alpha(t H_N)},
%\end{equation}
%where $\alpha(A)$ denotes the spectral abscissa of $A$, i.e.,
%$$
%\alpha(A) = \max\limits_{\lambda \in \sigma(A)} \textrm{Re} \, \lambda.
%$$
%Note:
%\begin{itemize}
% \item For the modified Bessel functions of the first kind, $\alpha(H_N) \leq 1$
% \item For the Bessel functions of the first kind, $\alpha(H_N) = 0$
%\end{itemize}
%
\begin{lem} \label{lem:element_bound}
Let $H_N$ be defined either as \eqref{eq:H_bessel} or \eqref{eq:H_bessel_modified}
and let $t>0$. Let $R\in\RR$ be any value such that $R>t$.
Then, the elements of $\ee^{tH_N}$ are bounded as
\begin{equation} \label{eq:element_bound2}
\left( \ee^{tH_N} \right)_{i,j} \leq  \, C(R) \, \lambda^{|i-j|},
\end{equation}
where $\lambda = \frac{t}{2R}$ and  
\begin{equation}\label{eq:CR}
C(R) = \max(\|\exp(tH_N)\|, 
 2 \sqrt{2} \frac{\ee^{R + \frac{1}{4R}}}{1 - \lambda}).
\end{equation}
%where $R>t$ can be chosen freely.
\end{lem}
\begin{proof}
We may apply directly the bound (3.10) in \cite[Sec.~3.7]{Benzi:2007:decay_bounds}.
We know that $t H_N$ has its spectrum inside the interval $[-t,t]$, 
which has the logarithmic capacity $\rho = t/2$.
For the integration contour we take the same ellipse as in \cite{Benzi:2007:decay_bounds},
so $V=2\pi$ and $M(R) = \ee^{R+\frac{1}{4R}}$, where $R>t$ can be chosen freely. 
Let $H_N = V D V^{-1}$ be the diagonalization of $H_N$. 
From Lemma~\ref{lem:lemma_bessel_exp_1} we know that $\kappa(V) = \sqrt{2}$.
The bound (3.10) of \cite{Benzi:2007:decay_bounds} gives \eqref{eq:element_bound2}.
\end{proof}

% \begin{lem} \label{lem:mn_integral}
% Let $m,n$ be integers greater than zero and let $t>0$. Then
% \begin{equation}
%  \int\limits_0^t  s^m (t-s)^n  \, \dd s = \frac{t^{m+n+1} \, m! \, n!}{(m+n+1)!}
% \end{equation}
% \begin{proof}
% With the change of variables $s \rightarrow x t$ we obtain
% $$
% \int\limits_0^t  s^m (t-s)^n \, \dd s= t^{m+n+1} \int\limits_0^1  x^m (1-x)^n \, \dd x = t^{m+n+1} B(m+1,n+1),
% $$
% where $B(x,y)$ is the Beta function (see~\cite[pp.\;258]{Stegun:1964:HANDBOOK}).
% For positive integers $n$ and $m$, $B(m+1,n+1) = \frac{ m! \, n!}{(m+n+1)!}$.
% \end{proof}
% \end{lem}

%We will need the following bound %(see also \eqref{Gallopoulos:1992:EFFICIENT}.
% $$
% \norm{r_N(A)} \leq \frac{\rho^{N+1}}{(N+1)!} \max \{1,\ee^{\mu(A)} \},
% $$
% where  $\mu(A) = \lambda_{\mathrm{max}}\left( \frac{A + A^*}{2} \right)$.
%for the $\varphi$ functions \eqref{def:phi_and_remainder}.
 \begin{lem} \label{lem:remainder_bound}
 For any matrix $A \in \CC^{n\times n}$ and positive integer $\ell$,
 $$
 \norm{ \varphi_\ell(A) }  \leq \frac{ \max(1,\ee^{\mu(A)})}{ \ell !},
 $$
where  $\mu(A)$ denotes the logarithmic norm, i.e., 
$\mu(A) = \max \{ \, \lambda \, : \, \lambda \in \Lambda(\frac{A + A^*}{2}) \}$.
%\lambda_{\mathrm{max}}\left( \frac{A + A^*}{2} \right)$.
 \end{lem}
 \begin{proof}
 From \eqref{def:phi_and_remainder} we see that
 $$
 \norm{ \varphi_\ell(A) } = \norm{\int\limits_0^1 \ee^{(1-t)A} \frac{t^{\ell-1}}{(\ell-1)!}} \, \dd t 
\leq \int\limits_0^1 \norm{\ee^{(1-t)A} } \frac{t^{\ell-1}}{(\ell-1)!} \, \dd t.
$$
%\ r_m^\ell(x) = 
% \norm{A^{m+1} \varphi_{m+ \ell +1}(A)} \leq \norm{A^{m+1}} \int\limits_0^1 
%\norm{\ee^{(1-\tau)A}} \frac{\tau^{k-1}} {(k-1)!} \, \dd \tau.
% $$
Using the Dahlquist bound $\norm{\ee^A} \leq \ee^{\mu(A)}$ and the fact that
$\mu((1-t)A) \leq \max\{0,\mu(A)\}$ for $ 0 \leq t \leq 1$, the claim follows.
 \end{proof}

\begin{lem} \label{lem:H_N_powers}
Let $H_N$ be defined as in \eqref{eq:H_bessel} or \eqref{eq:H_bessel_modified}. Then, for $k \geq N$,
\begin{equation} \label{eq:K_NH_Nbound}
\norm{K_N(H_N,e_1)^{-1}H_N^k e_1} \leq 2 \sqrt{N} (1+\sqrt{2})^N.
\end{equation}
\end{lem}
\begin{proof}
Let $p_N(\lambda) = \sum_{\ell=0}^N \alpha_\ell \lambda^\ell$ be the characteristic polynomial
of $H_N$.
%Let $\alpha_0, \ldots, \alpha_{N-1}$ be the first $N$ coefficients of the characteristic polynomial
%of $H_N$. 
Define
$$
 \wt \alpha_N = - \begin{bmatrix} \alpha_0 \\ \vdots \\ \alpha_{N-1} \end{bmatrix}, \quad \textrm{and} \quad
C(\wt \alpha_N) = \begin{bmatrix} & & & \\ 1 & & & \wt \alpha_N \\ & \ddots  & &  \\ & & 1 & \end{bmatrix}
\in \mathbb{R}^{N \times N}.
$$
Suppose $k\geq N$. Since $H_N^N e_1 = -\sum_{\ell=0}^{N-1} \alpha_\ell H_N^\ell e_1 = K_N(H_N,e_1) \wt \alpha_N$, we see that
$$
H_N K_N(H_N,e_1) = \begin{bmatrix} H_N e_1 & \ldots & H_N^N e_1 \end{bmatrix} = K_N(H_N,e_1) C(\wt \alpha_N),
$$
and since $H_N^{N-1} e_1 = K_N(H_N,e_1) e_N$, we see that
$$
H_N^k e_1 = H_N^{k-N+1} K_N(H_N,e_1) e_N = K_N(H_N,e_1) C(\wt \alpha_N)^{k-N+1} e_N,
$$
i.e., 
\begin{equation} \label{eq:K_C_connection}
K_N(H_N,e_1)^{-1} H_N^k e_1 = C(\wt \alpha_N)^{k-N+1} e_N.
\end{equation}
We recognize that $C(\wt \alpha_N)$ is the companion matrix of the $N$th Chebyshev polynomial $T_N$,
and that
$V_N C(\wt \alpha_N) = \Lambda_N V_N$,
where $\lambda_1,\ldots,\lambda_N$ are the zeroes of $T_N$, $\Lambda_N = \textrm{diag}(\lambda_1,\ldots,\lambda_N)$
and $V_N$ is the Vandermonde matrix corresponding to $\lambda_1,\ldots,\lambda_N$, i.e.,
$$
V_N = \begin{bmatrix} 1 & \lambda_1 & \hdots & \lambda_1^{N-1} \\ \vdots & \vdots & & \vdots \\ 1 & \lambda_N & \hdots & \lambda_N^{N-1} \end{bmatrix}.
$$
Thus for $\ell \geq 1$, $C(\wt \alpha_N)^\ell = V_N^{-1} \Lambda_N^\ell V_N$, and subsequently for any matrix norm $\norm{\cdot}_*$
\begin{equation} \label{eq:follows1}
\norm{C(\wt \alpha_N)^\ell}_* \leq \norm{V_N^{-1}}_* \norm{\Lambda^\ell}_* \norm{V_N}_* \leq 
\norm{V_N^{-1}}_* \norm{V_N}_* \quad \textrm{for all} \quad \ell \geq 1,
\end{equation}
since $\abs{\lambda_i}\leq 1$ for all  $1 \leq i \leq N$. From~\cite[Thm.\;4.3 and Example 6.2]{Gautschi}, we know that 
\begin{equation} \label{eq:follows2}
\norm{V_N^{-1}}_\infty \norm{V_N}_\infty  \leq 2 ( 1 + \sqrt{2})^N.
\end{equation}
Thus, using \eqref{eq:K_C_connection},  we see that for $k \geq N$
$\norm{K_N(H_N,e_1)^{-1}H_N^k e_1} \leq\norm{ C(\wt \alpha_N)^{k-N+1}} \leq \sqrt{N} \norm{ C(\wt \alpha_N)^{k-N+1}}_\infty$
% \leq 2 \sqrt{N} ( 1 + \sqrt{2})^N.
and the statement follows from \eqref{eq:follows1} and \eqref{eq:follows2}.
\end{proof}

\end{document}